\documentclass[11pt,noindent]{article}
\usepackage{latexsym,amsmath,amsfonts,hyperref}
\usepackage{graphics,color}

\usepackage{bbm}

\newtheorem{theorem}{Theorem}\numberwithin{theorem}{section}
\numberwithin{proposition}{section}
\newtheorem{corollary}{Corollary}\numberwithin{corollary}{section}
\newtheorem{lemma}{Lemma}\numberwithin{lemma}{section}

\newtheorem{definition}{Definition}

\newtheorem{example}{Example}\numberwithin{example}{section}
\newtheorem{claim}{Claim}

\newtheorem{remark}{Remark}

\newcommand{\thmref}[1]{Theorem~\ref{thm:#1}} 
\newcommand{\lemref}[1]{Lemma~\ref{lem:#1}} 
\newcommand{\remref}[1]{Remark~\ref{rem:#1}} 
\newcommand{\corref}[1]{Corollary~\ref{cor:#1}} 
\newcommand{\defref}[1]{Definition~\ref{def:#1}} 
\newcommand{\secref}[1]{Section~\ref{sec:#1}} 
\newcommand{\eqnref}[1]{(\ref{eq:#1})} 

\def\be{\begin{equation} }
\def\ee{ \end{equation}}

\def\ben{\begin{equation*}}
\def\een{\end{equation*}}
\def\bea{\begin{eqnarray}}
\def\eea{\end{eqnarray}}
\def\ee{\end{eqnarray}}
\def\bean{\begin{eqnarray*}}
\def\eean{\end{eqnarray*}}

\def\haslaw{\stackrel{\cal L}{\sim}}

\newcommand\ignore[1]{}

\def\R{\mathbb{R}} 
\def\N{\mathbb{N}} 

\def\EXP{\mathbb{E}} 
\def\PROB{\mathbb{P}} 

\newcommand{\Ex}[1]{\mathbb{E}\left[#1\right]} 

\newcommand{\Ind}[1]{{\mathbbm 1}\left\{#1\right\}} 

\renewcommand{\Pr}[1]{\mathbb{P}\left(#1\right)} 

\newcommand{\paren}[1]{\left(#1\right)}
\newcommand{\absj}[1]{\left|#1\right|}
\newcommand{\set}[1]{\left\{#1\right\}}

\newcommand{\bigoh}[1]{O\left(#1\right)}
\newcommand{\liloh}[1]{o\left(#1\right)}

\def\sG{\mathcal{G}}

\def\sP{\mathcal{P}}\def\sR{\mathcal{R}}

\newcommand\QED{\ifhmode\allowbreak\else\nobreak\fi
\quad\nobreak$\Box$\medbreak}
\newcommand{\proofstart}{\par\noindent\sl Proof:\rm\enspace}
\newcommand{\proofend}{\QED\par}
\newenvironment{proof}{\proofstart}{\proofend}

\def\eps{\epsilon}

\newcommand{\Po}{\mbox{\sf Po}}

\def\Po{{\rm Po}}

\def\Pp{{\rm P}}

\def\Pphat{\widehat{\rm P}}

\def\Ehat{\widehat{E}}
\def\Rhat{\widehat{R}}
\def\Ehatdelta{\widehat{E}_{n,\delta}}
\def\ahat{\hat{a}}
\def\bhat{\hat{b}}
\def\Ihat{\widehat{I}}
\def\muhat{\widehat{\mu}}
\def\shat{\widehat{\sigma}}
\def\nuhat{\widehat{\nu}}
\def\ellhat{\hat{\ell}}

\def\Po{{\sf Po}}
\def\La{{\sf La}}
\def\Lreg{L_*}

\def\Emphat{\widehat{\sf emp}_n}

\def\single{$\delta$-dependent}
\def\mult{multiple-$\delta$}

\begin{document}
\title{Sub-Gaussian mean estimators}
\author{Luc Devroye \and Matthieu Lerasle \and G\'{a}bor Lugosi \and Roberto I. Oliveira}
\maketitle
\begin{abstract}
We discuss the possibilities and limitations of estimating the mean of
a real-valued random variable from independent and identically
distributed observations from a non-asymptotic point of view. In
particular, we define estimators with a sub-Gaussian behavior even for
certain heavy-tailed distributions. We also prove various
impossibility results for mean estimators.
\end{abstract}

\section{Introduction}\label{sec:intro}

Estimating the mean of a probability distribution $\Pp$ on the real line based on a sample 
$X_1^n=(X_1,\dots,X_n)$ of $n$ independent and identically
distributed random variables is arguably the most basic problem of statistics. 
While the standard empirical mean
$$\Emphat(X_1^n)=\frac{1}{n}\sum_{i=1}^nX_i$$
is the most natural choice, its finite-sample performance is far from optimal 
when the distribution has a heavy tail. 

The central limit theorem guarantees that if the $X_i$ have a finite second moment,
this estimator has Gaussian tails, asymptotically, when $n\to \infty$. Indeed,   
\begin{equation}\label{eq:CLT}\Pr{\left|\Emphat(X_1^n) - \mu_{\Pp}\right|>\frac{\sigma_{\Pp}\,\Phi^{-1}(1-\delta/2)}{\sqrt{n}}}\to \delta,\end{equation}
where $\mu_{\Pp}$ and $\sigma^2_{\Pp}>0$ are the mean and variance of $\Pp$ (respectively) and $\Phi$ is the cumulative distribution function of the standard normal distribution. This result is essentially optimal: no estimator can have better-than-Gaussian tails for all distributions in any ``reasonable class"~(cf.\ \remref{optimal} below).  

This paper is concerned with a non-asymptotic version of the mean estimation problem. We are interested in large, non-parametric classes of distributions, such as
\begin{eqnarray}\label{eq:defP2}\sP_{ 2}&:=& \{\mbox{all distributions over $\R$ with finite second moment}\}\\
 \label{eq:deffixedvariance}\sP^{\sigma^2}_{2}&:=& \{\mbox{all distributions $\Pp\in\sP_2$ with variance $\sigma^2_{\Pp}=\sigma^2$}\} \;\;(\sigma^2>0)\\ \label{eq:defP4}
\sP_{{\rm krt}\leq \kappa}&:=& \{\mbox{all $\Pp\in\sP_2$ with kurtosis $\leq \kappa$}\}\;\;(\kappa>1),\end{eqnarray}
as well as some other classes introduced in \secref{results}. Given such a class $\sP$, we would like to construct {\em sub-Gaussian estimators}. These should take an i.i.d.\ sample $X_1^n$ from some unknown $\Pp\in\sP$ and produce an estimate $\Ehat_n(X_1^n)$ of $\mu_{\Pp}$ that satisfies
\begin{equation}\label{eq:sub-Gaussian-1}
\Pr{|\Ehat_n(X_1^n) - \mu_{\Pp}|>L\,\sigma_{\Pp}\,\sqrt{\frac{(1+\ln(1/\delta))}{n}}}\leq \delta
\quad \text{for all $\delta\in[\delta_{\min},1)$}
\end{equation}
for some constant $L>0$ that depends only on $\sP$. One would like to keep $\delta_{\min}$ as small as possible (say exponentially small in $n$). 

Of course, when $n\to\infty$ with $\delta$ fixed, \eqnref{sub-Gaussian-1} is a weaker form of \eqnref{CLT} since $\Phi^{-1}(1-\delta/2)\leq \sqrt{2\ln(2/\delta)}$. The point is that \eqnref{sub-Gaussian-1} should hold non-asymptotically, for extremely small $\delta$, and uniformly over $\Pp\in\sP$, even for classes $\sP$ containing distributions with heavy tails. The empirical mean cannot satisfy this property 
unless either $\sP$ contains only sub-Gaussian distributions or $\delta_{min}$ is quite
large  (cf.\ \secref{empirical}), so designing sub-Gaussian estimators with the kind of guarantee we look for is a non-trivial task.

In this paper we prove that, for most (but not all) classes
$\sP\subset \sP_2$ we consider, there do exist estimators that achieve \eqnref{sub-Gaussian-1} for all large $n$, with $\delta_{\min}\approx e^{-c_{\sP}\,n}$ and a value of $L$ that does not depend on $\delta$ or $n$. In each case, $c_{\sP}>0$ is a constant that depends on the class $\sP$ under consideration, and we also obtain nearly tight bounds on how $c_{\sP}$ must depend on $\sP$. (In particular, $\delta_{\min}$ cannot be superexponentially small in $n$.) In the specific case of bounded-kurtosis distributions (cf.\ \eqnref{defP4} above), we achieve  $L\leq \sqrt{2}+\eps$ for $\delta_{\min}\approx e^{-\liloh{(n/\kappa)^{2/3}}}$. This value of $L$ is nearly optimal by \remref{optimal} below. 

Before this paper, it was known that \eqnref{sub-Gaussian-1} could be
achieved for the whole class $\sP_2$ of distributions with finite
second moments, with a weaker notion of estimator that we call  {\em
  \single~estimator}, that is, an estimator $\Ehat_n=\Ehatdelta$ that
may also depend on the confidence parameter $\delta$. By contrast, the
estimators that we introduce here are called {\em \mult~estimators}: a
single estimator works for the whole range of $\delta\in
[\delta_{\min},1)$. 
This distinction is made formal in \defref{sub-Gaussian} below. By way of comparison, we also prove some results on \single~estimators in the paper. In particular, we show that the distinction is substantial. For instance, there are no \mult~sub-Gaussian estimators for the full class $\sP_2$ for any nontrivial range of $\delta_{\min}$. Interestingly,  \mult~estimators do exist (with $\delta_{\min}\approx e^{-c\,n}$) for the class $\sP_2^{\sigma^2}$ (corresponding to fixed variance). In fact, this is true when the variance is ``known up to constants," but not otherwise. 

\subsubsection*{Why finite variance?}
In all examples mentioned above, we assume that all distributions
$P\in\sP$ have a finite variance $\sigma_P^2$. In fact, our definition
(\ref{eq:sub-Gaussian-1}) implicitly requires that the variance exists
for all $P\in \sP$. A natural question is if this condition can be
weakened. For example, for any $\alpha \in (0,1]$ and $M>0$, one may consider the class 
$\sP^M_{1+\alpha}$ of all distributions whose $(1+\alpha)$-th central
moment equals $M$ (i.e., $\EXP\left[|X-\EXP X|^{1+\alpha}\right]=M$ if
$X$ is distributed according to any $P\in \sP^M_{1+\alpha}$). It is
natural to ask whether there exist estimators of the mean satisfying 
(\ref{eq:sub-Gaussian-1}) with $\sigma_P$ replaced by some constant 
depending on $P$. In Theorem \ref{thm:infvar} we prove that for every sample
size $n$, $\delta<1/2$, $\alpha\in (0,1]$, and for any mean estimator $\Ehatdelta$, 
there exists a distribution $P\in \sP^M_{1+\alpha}$ such that with
probability at least $\delta$, the estimator is at least 
$M^{1/(1+\alpha)}
\left(\frac{\ln(1/\delta)}{n}\right)^{\alpha/(1+\alpha)}$ away from
the target $\mu_{\Pp}$. 

This result not only shows that one
cannot expect sub-Gaussian confidence intervals for classes that
contain distributions of infinite variance but also that in such cases
it is impossible to have confidence intervals whose length scales as $n^{-1/2}$.

\subsubsection*{Weakly sub-Gaussian estimators}

Consider the class $\sP^{\text{Ber}}$ of all Bernoulli distributions,
that is, the class that contains all distributions $P$ of the form
\[
   P(\{1\})= 1-P(\{0\}) = p~,   \qquad  p\in [0,1]~.
\]
Perhaps surprisingly, no \mult~estimator exists for this class of
distributions, even when $\delta_{\min}$ is a constant. (We do not
explicitly prove this here but it is easy to deduce it using the
techniques of Sections \ref{sec:scaledBer} and \ref{sec:poisson}.)
On the other hand, by standard tail bounds for the binomial
distribution
(e.g., by Hoeffding's inequality),
the standard empirical mean satisfies, for all $\delta>0$ and $P\in \sP^{\text{Ber}}$,
\[
\Pr{|\Emphat(X_1^n) - \mu_{\Pp}|>\sqrt{\frac{\ln(2/\delta))}{2n}}}\leq \delta~.
\]
Of course, this bound has a sub-Gaussian flavor as it resembles
(\ref{eq:sub-Gaussian-1}) except that the confidence bounds do not
scale by $\sigma_{\Pp}(\log(1/\delta)/n)^{1/2}$ but rather by a distribution-free
constant times $(\log(1/\delta)/n)^{1/2}$.

In general, we may call an estimate weakly sub-Gaussian with respect
to the class $\sP$ if there exists a constant $\overline{\sigma}_{\sP}$ such that
for all $\Pp\in \sP$,
\[
\Pr{|\Ehat_n(X_1^n) - \mu_{\Pp}|>L\,\overline{\sigma}_{\sP}\,\sqrt{\frac{(1+\ln(1/\delta))}{n}}}\leq \delta
\quad \text{for all $\delta\in[\delta_{\min},1)$}
\]
for some constant $L>0$.
\single~and \mult~versions of this definition may be given in analogy
to those of sub-Gaussian estimators.

Note that if a class $\sP$ is such that $\sup_{\Pp\in \sP}
\sigma_{\Pp} < \infty$, then any sub-Gaussian estimator is weakly
sub-Gaussian. However, for classes of distributions without uniformly
bounded variance, this is not necessarily the case and the two notions
are incomparable. 

In this paper we focus on the
notion of sub-Gaussian estimators and we do not pursue further the
characterization of the existence of weakly  sub-Gaussian estimators.

\subsection{Related work} To our knowledge, the explicit distinction
between \single~and \mult~estimators, and our construction of
\mult~sub-Gaussian estimators for exponentially small $\delta$, are
all new. On the other hand, constructions of \single~estimators are
implicit in older work on stochastic optimization of Nemirovsky and
Yudin \cite{NemirovskyYudin_ProblemComplexity} (see also Levin
\cite{Levin_Notes} and Hsu \cite{HsuBlog}),
sampling from large discrete structures by Jerrum, Valiant, and
Vazirani \cite{JeVaVa86},
 and sketching algorithms, see Alon, Matias, and Szegedy \cite{AlonMatiasSzegedy_FreqMoments}. 
Recently, there has been a surge of interest in sub-Gaussian estimators, 
their generalizations to multivariate settings, and their applications in a variety
of statistical learning problems where heavy-tailed distributions may
be present, see, for example, 
Catoni \cite{Catoni_Challenging},
Hsu and Sabato \cite{HsuSabato_MoMRegression},
Brownlees, Joly, and Lugosi \cite{BrownleesEtAlHeavyTailed},
Lerasle and Oliveira \cite{LerasleOliveira_Robust},
Minsker \cite{Min13},
Audibert and Catoni \cite{AudibertCatoni_Robust},
Bubeck, Cesa-Bianchi, and Lugosi \cite{BubeckEtAl_BanditsHeavy}.
Most of these papers use \single~sub-Gaussian estimators. 
Catoni's paper \cite{Catoni_Challenging} is close in spirit to ours, as it focuses on sub-Gaussian mean estimation as a fundamental problem. That paper presents \single~sub-Gaussian estimators with nearly optimal $L=\sqrt{2}+\liloh{1}$ for a wide range of $\delta$ and the classes $\sP_2^{\sigma^2}$ and $\sP_{{\rm krt}\leq \kappa}$ defined in \eqnref{deffixedvariance}. 
The \single~sub-Gaussian estimator introduced by \cite{Catoni_Challenging} 
may be converted into a \mult~estimators with subexponential (instead of sub-Gaussian) tails for $\sP_2^{\sigma^2}$
by choosing the single parameter of the estimator appropriately. 
Loosely speaking, this corresponds to squaring the term $\ln(1/\delta)$ in \eqnref{sub-Gaussian-1}. 
Catoni also obtains \mult~estimators for $\sP_2$ with subexponential tails. These ideas are strongly related to Audibert and Catoni's paper on robust least-squares linear regression \cite{AudibertCatoni_Robust}. \\

\subsection{Main proof ideas}

The {\em negative results} we prove in this paper are minimax lower
bounds for simple families of distributions such as scaled Bernoulli distributions
(\thmref{infvar}), Laplace distributions with fixed scale parameter
for \single~(\thmref{lowerfixed}), and the Poisson family for
\mult~estimators (\thmref{negativestrong}). The main point about the
latter choices is that it is easy to compare the probabilities of events when one changes the values of the parameter. Interestingly, Catoni's lower bounds in \cite{Catoni_Challenging} also follow from a one dimensional family (in that case, Gaussians with fixed variance $\sigma^2>0$).

Our {\em constructions of estimators} use two main ideas. The first one is that, while one cannot turn \single~into \mult~estimators, one {\em can} build \mult~estimators from the slightly stronger concept of {\em sub-Gaussian confidence intervals}. That is, if for each $\delta>0$ one can find an empirical confidence interval for $\mu_{\Pp}$ with ``sub-Gaussian length", one may combine these intervals to produce a single \mult~estimator. This general construction is presented in \secref{confidence} and is related at a high level to Lepskii's adaptation method \cite{Lep90,Lep91}. 

Although general, this method of confidence intervals loses constant
factors. Our second idea for building estimators, which is specific to the bounded kurtosis case (see  \thmref{OptConstUnderKurtosis} below), is to use a data-driven truncation mechanism to make the empirical mean better behaved. By using preliminary estimators of the mean and variance, we truncate the random variables in the sample and obtain a Bennett-type concentration inequality with sharp constant $L=\sqrt{2}+\liloh{1}$. A crucial point in this analysis is to show that our truncation mechanism is fairly insensitive to the preliminary estimators being used.  
 
\subsection{Organization.}

The remainder of the paper is organized as follows. \secref{prelim} fixes notation, formally defines our problem, and discusses previous work in light of our definition. \secref{results} states our main results. Several general methods that we use throughout the paper are collected in \secref{GeneralMethod}. Proofs of the main results are given in Sections \ref{sec:variance} to \ref{sec:BoundedKurtosis}. \secref{open} discusses several open problems. 

\section{Preliminaries} \label{sec:prelim}

\subsection{Notation}

We write $\N=\{0,1,2,\dots\}$. For a positive integer $n$, denote $[n]=\{1,\ldots,n\}$. 
$|A|$ denotes the cardinality of the finite set $A$.

We treat $\R$ and $\R^n$ as measurable spaces with the respective Borel $\sigma$-fields kept implicit. Elements of $\R^n$ are denoted by $x_1^n=(x_1,\dots,x_n)$ with $x_1,\dots,x_n\in\R$.

Probability distributions over $\R$ are denoted $\Pp$. Given a (suitably measurable) function $f=f(X,\theta)$ of a real-valued random variable $X$ distributed according to 
$\Pp$ and some other parameter $\theta$, we let $$\Pp\,f=\Pp\,f(X,\theta)=\int_{\R}f(x,\theta)\,\Pp(dx)$$ denote the integral of $f$ with respect to $X$. Assuming $\Pp\,X^2<\infty$, we use the symbols $\mu_{\Pp}=\Pp \,X$ and $\sigma^2_{\Pp}=\Pp\,X^2-\mu^2_{\Pp}$ for the mean and variance of $\Pp$. 

$Z=_d\Pp$ means that $Z$ is a random object (taking values in some measurable space) and $\Pp$ is the distribution of this object. $X_1^n=_ d\Pp^{\otimes n}$ means that $X_1^n=(X_1,\dots,X_n)$ is a random vector in $\R^n$ with the product distribution corresponding to $\Pp$. Moreover, given such a random vector $X_1^n$ and a nonempty set $B\subset [n]$, $\Pphat_B$ is the empirical measure of $X_i$, $i\in B$:
$$\Pphat_B=\frac{1}{|B|}\sum_{i\in B}\delta_{X_i}.$$
We write $\Pphat_n$ instead of $\Pphat_{[n]}$ for simplicity.

\subsection{The sub-Gaussian mean estimation problem}

In this section, we begin a more formal discussion of the main problem
in this paper. We start with the definition of a sub-Gaussian
estimator of the mean.

\begin{definition}
\label{def:sub-Gaussian}
Let $n$ be a positive integer, $L>0$, $\delta_{\min}\in (0,1)$. Let
$\sP$ be a family of probability distributions over $\R$ with finite second moments. 
\begin{enumerate}\item {\bf \single~sub-Gaussian estimation:} a
  \single~$L$-sub-Gaussian estimator for $(\sP,n,\delta_{\min})$ is a
  measurable mapping $\Ehat_n:\R^n\times [\delta_{\min},1)\to \R$ such
  that if $\Pp\in\sP$,  $\delta\in[\delta_{\min},1)$, and
 $X_1^n=(X_1,\ldots,X_n)$ is a sample of i.i.d.\ random variables
  distributed as $\Pp$, then
\begin{equation}\label{eq:sub-Gaussianfixed}\Pr{|\Ehat_n(X_1^n,\delta) - \mu_{\Pp}|>L\,\sigma_{\Pp}\,\sqrt{\frac{(1+\ln(1/\delta))}{n}}}\leq \delta~.\end{equation}
We also write $\Ehatdelta(\cdot)$ for $\Ehat_n(\cdot,\delta)$.
\item {\bf \mult~sub-Gaussian estimation:} a \mult~$L$-sub-Gaussian
  estimator for $(\sP,n,\delta_{\min})$ is a measurable mapping
  $\Ehat_n:\R^n\to \R$ such that, for each $\delta\in
  [\delta_{\min},1)$, $\Pp\in\sP$ and 
i.i.d.\ sample $X_1^n=(X_1,\ldots,X_n)$  distributed as $\Pp$,
\begin{equation}\label{eq:sub-Gaussian}\Pr{|\Ehat_n(X_1^n) - \mu_{\Pp}|>L\,\sigma_{\Pp}\,\sqrt{\frac{(1+\ln(1/\delta))}{n}}}\leq \delta~.\end{equation}
\end{enumerate}\end{definition}

It transpires from these definitions that \mult~estimators are
preferable whenever they are available, because they combine good
typical behavior with nearly optimal bounds under extremely rare
events. By contrast, the need to commit to a $\delta$ in advance means
that \single~estimators may be too pessimistic when a small $\delta$
is desired. The main problem addressed in this paper is the following:

\medskip
\noindent
\framebox{\parbox{\dimexpr\linewidth-2\fboxsep-2\fboxrule}{\itshape%
Given a family $\sP$ (or more generally a sequence of families
$\sP_n$), find the smallest possible sequence
$\delta_{\min}=\delta_{\min,n}$ such that \mult~$L$-sub-Gaussian
estimators for $(\sP,n,\delta_{\min,n})$
(resp. $(\sP_n,n,\delta_{\min,n})$) exist for all large $n$, and with
a constant $L$ that does not depend on $n$.}}

\begin{remark}
\textsc{(optimality of sub-gaussian estimators.)}
\label{rem:optimal} Call a class $\sP$ ``reasonable"~when it contains
all Gaussian distributions with a given variance $\sigma^2>0$. Catoni \cite[Proposition 6.1]{Catoni_Challenging} shows that, if $\delta\in(0,1)$, $\sP$ is reasonable and some estimator $\Ehatdelta$ achieves 
$$\Pr{\Ehatdelta(X_1^n)-\mu_{\Pp}>\frac{r\,\sigma_{\Pp}}{\sqrt{n}}}\leq \delta\mbox{ whenever }\Pp\in\sP~,$$
then $r\geq \Phi^{-1}(1-\delta)$. The same result holds for the lower tail. Since $\Phi^{-1}(1-\delta)\sim \sqrt{2\ln(1/\delta)}$ for small $\delta$, this means that, for any reasonable class $\sP$, no constant $L<\sqrt{2}$ is achievable for small $\delta_{\min}$, and no better dependence on $n$ or $\delta$ is possible. In particular, sub-Gaussian estimators are optimal up to constants, and estimators with $L\leq \sqrt{2}+o(1)$ are ``nearly optimal."\end{remark}

\subsection{Known examples from previous work}

In what follows we present some known estimators of the mean and discuss their sub-Gaussian properties (or lack thereof).

\subsubsection{Empirical mean as a sub-Gaussian
  estimator}\label{sec:empirical} For large $n$, $\sigma^2>0$ fixed
and $\delta_{\min}\to 0$, the empirical mean 
$$\Emphat(x_1^n)=\frac{1}{n}\sum_{i=1}^nx_i$$
 is {\em not} a $L$-sub-Gaussian estimator for the class $\sP^{\sigma^2}_2$ of all distibutions with variance $\sigma^2$. This is a consequence of \cite[Proposition 6.2]{Catoni_Challenging}, which shows that the deviation bound obtained from Chebyshev's inequality is essentially sharp. 

Things change under slightly stronger assumptions. For example, a
nonuniform version of the Berry-Ess\'{e}en theorem (\cite[Theorem 14,
p.\ 125]{Pet75})
implies that, for large $n$, $\Emphat$ is a  \mult~$\left(\sqrt{2}+\eps\right)$-sub-Gaussian estimator for $(\sP_{3,\eta},n,\delta_{\min,n})$, where
$$\sP_{3,\eta}=\{\Pp\in\sP_{2}\,:\,\Pp|X-\mu_{\Pp}|^3\leq (\eta\,\sigma)^3\}$$
for some $\eta>1)$ and $\delta_{\min,n}\gg n^{-1/2}(\log
n)^{-3/2}$. Similar results (with worse constants) hold for the class
$\sP_{{\rm krt}\leq \kappa}$ (cf.\ \eqnref{defP4}) when
$\delta_{\min}\gg 1/n$ and $\kappa$ is bounded \cite[Proposition
5.1]{Catoni_Challenging}. Catoni \cite[Proposition
6.3]{Catoni_Challenging} shows that the sub-Gaussian property breaks
down when $\delta_{\min}=\liloh{1/n}$. Exponentially small
$\delta_{\min}$ can be achieved under much stronger assumptions. For example, Bennett's inequality implies that $\Emphat$ is $\left(\sqrt{2}+\eps\right)$-sub-Gaussian for the triple $(\sP_{\infty,\eta},n,\delta_{\min})$, with $\delta_{\min}=e^{-\eps^2n/\eta^2}$ and
$$\sP_{\infty,\eta}:=\{\Pp\in\sP_{2}\,:\,|X-\mu_{\Pp}|\leq \eta\,\sigma_{\Pp}\mbox{ a.s.}\}~.$$

\subsubsection{Median of means}\label{sec:MoM} 
Quite remarkably, as it has been known for some time, one can do much
better than the empirical mean in the \single~setting. The so-called
{\em median of means} construction gives $L$-sub-Gaussian estimators
$\Ehatdelta$ (with $L$ some constant) for any triple
$(\sP_2,n,e^{1-n/2})$ where $n\geq 6$. The basic idea is to partition
the data into disjoint blocks, calculate the empirical mean within
each block, and finally take the median of them.
This construction with a basic performance bound is reviewed
in \secref{MoM2}, as it provides a building block and an inspiration
for the new constructions in this paper. We emphasize that, as pointed
out in the introduction, variants of this result have been known for a
long time, see
Nemirovsky and Yudin \cite{NemirovskyYudin_ProblemComplexity}, 
Levin \cite{Levin_Notes},
Jerrum, Valiant, and Vazirani \cite{JeVaVa86},
and Alon, Matias, and Szegedy \cite{AlonMatiasSzegedy_FreqMoments}. 
Note that this estimator has good performance even for distributions with infinite
variance
(see the remark following Theorem \ref{thm:infvar} below).

\subsubsection{Catoni's estimators}\label{sec:Catoni}
The constant $L$ obtained by the median-of-means estimator
is larger than the optimal value $\sqrt{2}$ (see \remref{optimal}). 
Catoni \cite{Catoni_Challenging} designs \single~sub-Gaussian estimators with
nearly optimal $L=\sqrt{2}+o(1)$ for the classes
$\sP_{2}^{\sigma^2}$ (known variance) and $\sP_{{\rm krt}\leq \kappa}$
(bounded kurtosis). A variant of Catoni's estimator is a
\mult~estimator, however with subexponential instead of sub-Gaussian
tails (i.e., the $\sqrt{\ln(1/\delta)}$
term in \eqnref{sub-Gaussian} appears squared). Both estimators work
for exponentially small $\delta$, although the constant in the
exponent for $\sP_{{\rm krt}\leq \kappa}$ depends on $\kappa$.

\section{Main results}\label{sec:results}

Here we present the main results of the paper. 
Proofs are deferred to Sections~\ref{sec:GeneralMethod} to \ref{sec:BoundedKurtosis}. 

\subsection{On the non-existence of sub-Gaussian mean estimators}

Recall that for any $\alpha,M >0$, $\sP^M_{1+\alpha}$ denotes the class
of all distributions on $\R$ 
whose $(1+\alpha)$-th central
moment equals $M$ (i.e., $\EXP\left[|X-\EXP X|^{1+\alpha}\right]=M$). 
We start by pointing out that when $\alpha <1$, no sub-Gaussian estimators exist 
(even if one allows \single~estimators).

\begin{theorem}
\label{thm:infvar}
Let $n>5$ be a positive integer, $M>0$, $\alpha \in (0,1]$, and $\delta
\in (2e^{-n/4},1/2)$. Then for any mean estimator $\Ehat_n$,
\[
\sup_{P\in \sP^M_{1+\alpha}} 
\Pr{|\Ehat_n(X_1^n,\delta) - \mu_{\Pp}|> \left(\frac{M^{1/\alpha}\ln(2/\delta)}{n}\right)^{\alpha/(1+\alpha)}}\geq \delta~.
\]
\end{theorem}

The proof is given in Section \ref{sec:scaledBer}.
The bound of the theorem is essentially tight. It is shown in 
Bubeck, Cesa-Bianchi, and Lugosi \cite{BubeckEtAl_BanditsHeavy} that
for each $M>0$, $\alpha \in (0,1]$, and $\delta$, there exists
an estimator $\Ehat_n(X_1^n,\delta)$ such that
\[
\sup_{P\in \sP^M_{1+\alpha}} 
\Pr{|\Ehat_n(X_1^n,\delta) - \mu_{\Pp}|> \left(8\frac{(12M)^{1/\alpha}\ln(1/\delta)}{n}\right)^{\alpha/(1+\alpha)}}\le \delta~.
\]
The estimator $\Ehat_n(X_1^n,\delta)$ satisfying this bound is the
median-of-means estimator with appropriately chosen parameters.

It is an interesting question whether \mult~estimators exist with 
similar performance. Since our primary goal in this paper is the study
of sub-Gaussian estimators, we do not pursue the case of infinite
variance further.

\subsection{The value of knowing the variance} 
Given $0<\sigma_1\leq \sigma_2<\infty$, define the class of
distributions with variance between $\sigma_1^2$ and $\sigma_2^2$:
$$\sP_2^{[\sigma^2_1,\sigma^2_2]}=\{\Pp\in\sP_2\,:\,\sigma_1^2\leq \sigma^2_{\Pp}\leq \sigma^2_2.\}$$
This class interpolates between the classes of distributions with
fixed variance $\sP^{\sigma^2}_2$ and with completely unknown variance
$\sP_2$. The next theorem is proven in \secref{variance}.

\begin{theorem}
\label{thm:secondmoment}
Let $0<\sigma_1 < \sigma_2 < \infty$ and
 define $R=\sigma_2/\sigma_1$.
\begin{enumerate}
\item Letting $L^{(1)}=(4e\sqrt{2+4\ln 2})R$ and
  $\delta^{(1)}_{\min}=4e^{1-n/2}$, 
for every $n\ge 6$ there exists a \mult~$L^{(1)}$-sub-Gaussian estimator for $(\sP_2^{[\sigma^2_1,\sigma^2_2]},n,\delta^{(1)}_{\min})$.
\item For any $L\geq \sqrt{2}$, there exist $\phi^{(2)}>0$ and
  $\delta^{(2)}_{\min}>0$ such that, when $R>\phi^{(2)}$, there is no
  \mult~$L$-sub-Gaussian estimator for
  $(\sP_2^{[\sigma^2_1,\sigma^2_2]},n,\delta^{(2)}_{\min})$ for any $n$.
\item For any value of $R\geq 1$ and $L\geq \sqrt{2}$, if we let
  $\delta_{\min}^{(3)}=e^{1-5L^2n}$, there is no
  \single~$L$-sub-Gaussian estimator for
  $(\sP_2^{[\sigma^2_1,\sigma^2_2]},n,\delta^{(3)}_{\min})$ for any
  $n$.
\end{enumerate}
\end{theorem}

It is instructive to consider this result when $n$ grows and $R=R_n$
may change with $n$. The theorem says that, when $\sup_n R_n<\infty$,
there are \mult~$L$-sub-Gaussian estimators for all large $n$, with
exponentially small $\delta_{\min}$ and a constant $L$. On the other
hand, if $R_n\to \infty$, for any constant $L$ and all large $n$,  no
\mult~$L$-sub-Gaussian estimators exist for any sequence $\delta=\delta_{\min,n}\to 0$. Finally, the third item says that even when $R_n\equiv 1$, \single~estimators are limited to $\delta_{\min}=e^{-\bigoh{n}}$, so the median-of-means estimator is optimal in this sense.

\subsection{Regularity, symmetry and higher moments}\label{sec:regularresults}

\thmref{secondmoment} shows that finite, but completely unknown variance is too weak an assumption for \mult~sub-Gaussian estimation. The following shows that what we call {\em regularity conditions} can substitute for knowledge of the variance.

\begin{definition}\label{def:regular}
For $\Pp\in\sP_{\rm 2}$ and $j\in\N\backslash\{0\}$, let
$X_1,\ldots,X_j$ be i.d.d.\ random variables with distribution $\Pp$.
Define
$$p_-(\Pp,j)= \Pr{\sum_{i=1}^jX_i\leq j\mu_{\Pp}}\quad \mbox{ and }
\quad
p_+(\Pp,j)=\Pr{\sum_{i=1}^jX_i\geq  j\mu_{\Pp}}~.$$
Given $k\in\N$, we define the $k$-regular class as follows:

$$\sP_{{\rm 2},\,k{\rm -reg}}=\{\Pp\in\sP_{\rm 2}\,:\,\forall j\geq k,\,\min(p_+(\Pp,j),p_-(\Pp,j))\geq 1/3\}.$$\end{definition}

Note that this family of distributions is increasing in $k$. Also note
that $\bigcup_{k\in \N} \sP_{2,k{\rm -reg}}=\sP_2$, because the
central limit theorem implies $p_+(\Pp,j)\to 1/2$ and $p_-(\Pp,j)\to 1/2$. Here are two important examples of large families of distributions in this class:

\begin{example}\label{example:sym}
We say that a distribution
  $\Pp\in\sP_{\rm 2}$ is {\em symmetric around the mean} if, given
  $X=_d\Pp$, $2\mu_{\Pp}-X=_d\Pp$ as well. Clearly, if $\Pp$ has this
  property, $p_+(\Pp,j)=p_-(\Pp,j)=1/2$ for all $j$ and thus 
$\Pp\in \sP_{{\rm 2},1{\rm -reg}}$. In other words, 
$\sP_{2,{\rm sym}} \subset \sP_{2,1{\rm -reg}}$ where $\sP_{2,{\rm sym}}$ is
the class of all $\Pp\in \sP_2$ that are symmetric around the mean.
\end{example} 
\begin{example}\label{example:alphaclass}Given $\eta\geq 1$ and $\alpha\in(2,3]$, set
\begin{equation}\label{eq:defBEclass}\sP_{\alpha,\eta}=\{\Pp\in\sP_{\rm 2}\,:\,\Pp|X-\mu_{\Pp}|^{\alpha}\leq (\eta\,\sigma_{\Pp})^{\alpha}\}.\end{equation}
We show in \lemref{kregmoments} that, for $\Pp$ in this family,
$\min(p_+(\Pp,j),p_-(\Pp,j)) \geq 1/3$ once $j\geq
(C_\alpha\,\eta)^{\frac{2\alpha}{\alpha-2}}$ for a constant $C_\alpha$ depending only on $\alpha$. We deduce
$$\sP_{\alpha,\eta}\subset \sP_{{\rm 2},\,k{\rm -reg}}\mbox{ if } k\geq (C_\alpha\,\eta)^{\frac{2\alpha}{\alpha-2}}.$$\end{example}

Our main result about $k$-regular classes states that sub-Gaussian
\mult~estimators exist for $\sP_{2,k-{\rm reg}}$ in the sense of the following
theorem, proven in \secref{regular_main}.

\begin{theorem}
\label{thm:regular}
Let $n,k$ be positive integers with $n\geq (3+\ln 4)\,124k$. Set $\delta_{\min,n,k}=4e^{3-n/(124 k)}$ and $\Lreg=4\sqrt{2\,(1+2\ln 2)\,(1+62\ln(3))}\,e^{\frac{5}{2}}$. Then there exists a $\Lreg$-sub-Gaussian \mult~estimator for $(\sP_{2,k-{\rm reg}},n,\delta_{\min,n,k})$.\end{theorem}

We also show that the range of $\delta_{\min}=e^{-\bigoh{n/k}}$ in
this result is optimal. This follows directly from stronger results
that we prove for Examples \ref{example:sym} and
\ref{example:alphaclass}. In other words, the general family of
estimators designed for $k$-regular classes has nearly optimal range
of $\delta$ for these two smaller classes.  The next result, for
symmetric distributions, is proven in
\secref{symmetric}.

\begin{theorem}
\label{thm:symmetric} Consider the class $\sP_{2,{\rm sym}}$ defined in
Example \ref{example:sym}. Then
\begin{enumerate}
\item the estimator obtained in \thmref{regular} for $k=1$ is a $\Lreg$-sub-Gaussian \mult~estimator for $(\sP_{2,{\rm sym}},n,\delta_{\min,n,1})$ when $n\geq (3+\ln 2)\,124$;
\item on the other hand, for any $L\geq \sqrt{2}$, no \single~$L$-sub-Gaussian estimator can exist for  
$(\sP_{2,{\rm sym}},n,e^{1-5L^2n})$.\end{enumerate}
\end{theorem}

We also have an analogue result for the class $\sP_{\alpha,\eta}$. The
proof may be found in \secref{alphaclass}.

\begin{theorem}
\label{thm:alphaclass} Fix $\alpha\in (2,3]$ and assume $\eta\geq
3^{1/3}\,2^{1/6}$. Consider the class $\sP_{\alpha,\eta}$ defined in Example \ref{example:alphaclass}. Then there exists some $C_\alpha>0$ depending only on $\alpha$ such that if $k_\alpha=\lceil C_\alpha\,\eta^{(2\alpha)/(\alpha-2)}\rceil$,
\begin{enumerate}
\item the estimator obtained in \thmref{regular} for $k=k_\alpha$ is a $\Lreg$-sub-Gaussian \mult~estimator for $(\sP_{\alpha,\eta},n,\delta_{\min,n,k_\alpha})$ when $n\geq (3+\ln 4)\,124k_\alpha$;
\item on the other hand, for any $L\geq \sqrt{2}$, there exist $n_{0,\alpha,L}\in\N$ and $c_{\alpha,L}>0$ such that no \mult~$L$-sub-Gaussian estimator can exist for  
$(\sP_{\alpha,\eta},n,e^{1-c_{\alpha,L}\,n/k_\alpha})$
when $n\geq n_{0,\alpha,L}$ is large enough;
\item finally, for $L\geq \sqrt{2}$ there is no \single~$L$ sub-Gaussian estimator for $(\sP_{\alpha,\eta},n,e^{1-5L^2n})$.\end{enumerate}\end{theorem}

\subsection{Bounded kurtosis and nearly optimal constants}\label{sec:kurtosis}
This section shows that \mult~sub-Gaussian estimation with nearly optimal constants can be proved when the kurtosis 
\[\kappa_{\Pp}=\frac{\EXP (X-\mu_{\Pp})^4}{\sigma_{\Pp}^4}
\]
(when $X=_d \Pp$) is uniformly bounded in the class.
(For completeness, we set $\kappa_{\Pp}=1$ when $\sigma^2_{\Pp}=0$.) More specifically, we will consider the class $\sP_{{\rm krt}\leq \kappa}$ of all distributions $\Pp\in\sP_{2}$ with $\kappa_{\Pp}\leq \kappa$.

To state the result, let $b_{\max}$ be a positive integer to be
specified below. Also define
\begin{align*}
 \xi=2\sqrt2\kappa\frac{b_{\max}^{3/2}}{n}+ 36\sqrt{\frac{\kappa b_{\max}}n}+1120\sqrt{\kappa}\frac{b_{\max}}n~.
\end{align*}
Note that when $b_{\max}\ll (n/\kappa)^{2/3}$, $\xi=o(1)$. 
The main result for classes of distributions with bounded kurtosis is
the following. For the proof see \secref{BoundedKurtosis}.
\begin{theorem}
\label{thm:OptConstUnderKurtosis}
Let $n\ge 4$, $L=\sqrt{2}(1+\xi)$,
$\delta_{\min}^{(4)}=\frac{4e}{e-2}e^{-b_{\max}}$. There exists
an absolute constant $C$ such that, if $\kappa b_{\max}/n\le C$,
then there exists a \mult~$L$-sub-Gaussian estimator for $(\sP_4^{\kappa},n,\delta_{\min}^{(4)})$.
\end{theorem}

This result is most interesting in the regime where $n\to \infty$, $\kappa=\kappa_n$ possibly depends on $n$ and $n/\kappa_n\to \infty$. In this case, we may take  $b_{\max}\ll (n/\kappa_n)^{2/3}$ and obtain \mult~$\left(\sqrt{2}+\liloh{1}\right)$-sub-Gaussian estimators $(\sP_{{\rm krt}\leq \kappa},n,\delta^{(4)}_{\min})$ for $\delta^{(4)}_{\min}\approx e^{-b_{\max}}$. Catoni \cite{Catoni_Challenging} obtained \single~$\sqrt{2}+\liloh{1}$-estimators for a smaller value $\delta^{(5)}_{\min}\approx e^{-n/\kappa}$. In \remref{smallerdeltamin} we show how one can obtain a similar range of $\delta$ with a \mult~estimator, albeit with worse constant $L$.

\section{General methods} \label{sec:GeneralMethod}

We collect here some ideas that recur in the remainder of the paper. 
\begin{enumerate}
\item \secref{MoM2} presents an analysis of the median-of-means estimator mentioned in \secref{MoM} above. We present a proof based on Hsu's argument \cite{HsuBlog}. 
\item \secref{confidence} presents a ``black-box method" of deriving \mult~estimators from confidence intervals. The point is that confidence intervals are ``\single~objects", and thus easier to design and analyze. 
\item In \secref{scaledBer} we use scaled Bernoulli distributions to
  prove the impossibility of designing (weakly) sub-Gaussian estimators for
  classes with distributions with unbounded variance. 
\item \secref{laplace} uses the family of Laplace distributions to lower bound $\delta_{\min}$ for \single~estimators. 
\item \secref{poisson} uses the Poisson family to derive lower bounds on $\delta_{\min}$ for \mult~estimators.
\end{enumerate}

A combination of the above results will allow us to derive the sharp range for $\ln(1/\delta_{\min})$ for all families of distributions we consider. 

\subsection{Median of means}\label{sec:MoM2}

The next result is a well known performance bound for the
median-of-means estimator. We include the proof for completeness.

\begin{theorem}\label{thm:MoM2}For any $n\geq 4$ and $L=2\sqrt{2}\,e$ there exists a \single~$L$-sub-Gaussian estimators for $(\sP_{\rm 2},n,e^{1-n/2}).$ \end{theorem} 
\begin{proof}We follow the argument of Hsu \cite{HsuBlog}. Given a
  positive integer $b$ and a vector $x_1^b\in\R^b$, we let $q_{1/2}$
  denote the median of the numbers $x_1,x_2,\dots,x_b$, that is, 
$$q_{1/2}(x_1^b) = x_i,\mbox{ where }\# \{k\in [b]\,:\,x_k\leq x_i\}\geq \frac{b}{2}\mbox{ and } \# \{k\in [b]\,:\,x_k\geq x_i\}\geq \frac{b}{2}.$$
(If several $i$ fit the above description, we take the smallest one.) We need the following Lemma (proven subsequently):

\begin{lemma}
\label{lem:median}
Let $Y_1^b=(Y_1,\dots,Y_b)\in\R^b$ be independent random variables with the same mean $\mu$ and variances bounded by $\sigma^2$. Assume $L_0>1$ is given
and $M_b=q_{1/2}(Y_1^b)$. Then $\Pr{|M_b-\mu|>2\,L_0\,\sigma}\leq L_0^{-b}$.\end{lemma}

In our case we set $L_0=e=L/2\sqrt{2}$. To build our estimator for a given $\delta\in [e^{1-n/2},1)$, we first choose $$b=\left\lceil \ln(1/\delta)\right\rceil$$ and note that $b\leq n/2$.
 
Now divide $[n]$ into $b$ {\em blocks} (i.e., disjoint subsets) $B_i$, $1\leq i\leq b$, each of size $|B_i|\geq k=\lfloor n/b\rfloor\geq 2$. Given $x_1^n\in\R^{n}$, we define
$$y_{n,\delta}(x_1^n) = (y_{n,\delta,i}(x_1^n))_{i=1}^b\in\R^b\mbox{ with coordinates }y_{n,\delta,i}(x_1^n)=\frac{1}{|B_i|}\sum_{j\in B_i}x_j.$$
and define the median-of-means estimator by $\Ehatdelta(x_1^n)= q_{1/2}(y_{n,\delta}(x_1^n)).$ 

We now show that $\Ehatdelta$ is a sub-Gaussian estimator for the
class $\sP_2$. Let $X_1^n=_d\Pp^{\otimes n}$ for a distribution $\Pp\in\sP_{2}$.  $\Ehatdelta(X_1^n)$ is the median of random variables 
$$Y_i=\frac{1}{|B_i|}\sum_{j\in B_i}\,X_j = \Pphat_{B_i}X \quad
\quad \quad i\in[b]~.$$
Each $Y_i$ has mean $\mu_{\Pp}$ and variance $\sigma^2_{\Pp}/\# B_i
\leq \sigma^2_{\Pp}/k$. Then, using our choice of $b$, Lemma \ref{lem:median} implies 
$$\Pr{|\Ehatdelta(X_1^n)-\mu_{\Pp}|>\frac{2\,L_0\sigma_{\Pp}}{\sqrt{k}}}\leq L_0^{-b}\leq \delta~.$$
Now, because $b=\lceil \ln(1/\delta)\rceil \leq n/2$
$$k=\left\lfloor\frac{n}{b}\right\rfloor\geq \frac{n}{b}-1\geq\frac{n}{2b} \geq \frac{n}{2(1+\ln(1/\delta))},$$ and
$$\frac{2\,L_0}{\sqrt{k}}\leq {\frac{2\,L_0\sqrt{2}\,\sqrt{1+\ln(1/\delta)}}{\sqrt{n}}}=L\sqrt{\frac{1+\ln(1/\delta)}{n}}.$$
Therefore,
$$\Pr{|\Ehatdelta(X_1^n)-\mu_{\Pp}|>L\,\sigma_{\Pp}\,\sqrt{\frac{\ln(1/\delta)}{n}}}\leq \delta~,$$
and since this works for any $\Pp\in\sP_{\rm 2}$, the proof is complete. \end{proof}

\medskip
\noindent
{\sl Proof of \lemref{median}:} 
Let $I=[\mu -2\,L_0\,\sigma,\mu+2\,L_0\sigma]$. Clearly, 
$$M_b\not\in I\Rightarrow \#\{j\in[b]\,:\,Y_j\not\in I\}\geq \frac{b}{2}\Rightarrow \sum_{j=1}^b\,\Ind{Y_j\not\in I}\geq \frac{b}{2}.$$
The indicators variables on the right-hand side are all independent,
and by Chebyshev's inequality,
for all $j\in[b]$,
$$\Pr{Y_j\not\in I} \leq \frac{\Ex{(Y_j-\mu)^2}}{4\,L^2_0\sigma^2}\leq \frac{1}{4L^2_0}.$$
We deduce that $\sum_{j=1}^b\,\Ind{Y_j\not\in I}$ is stochastically
dominated by a binomial random variable ${\rm Bin}(b,(2L_0)^{-2})$ and
therefore, 
\begin{align*}
\Pr{M_b\not\in I}&\leq \Pr{{\rm Bin}(b,(2L^0)^{-2})\geq \frac{b}{2}}=\sum_{k=\lceil b/2\rceil}^b\binom{b}{k}\left(\frac{1}{(2L_0)^2}\right)^{k }\left(1-\frac{1}{(2L_0)^2}\right)^{b-k }\\
&\leq \left(\frac{1}{(2L_0)^2}\right)^{\lceil b/2\rceil }\sum_{k=\lceil b/2\rceil}^b\binom{b}{k}\leq L_0^{-b} 
\end{align*}
since $\sum_{k=\lceil b/2\rceil}^b\binom{b}{k}\le \sum_{k=0}^b\binom
bk=2^b$. 

\subsection{The method of confidence intervals for \mult~estimators}\label{sec:confidence}

In this section we detail how sub-Gaussian confidence intervals may be
combined to produce \mult~estimators. This will be our main tool in
defining all \mult~estimators whose existence is claimed in Theorems
\ref{thm:secondmoment} and \ref{thm:regular}. 
First we need a definition. 

\begin{definition}
Let $n$ be a positive integer, $\delta\in (0,1)$ and let $\sP$ be a
class of probability distributions over $\R$. A measurable closed interval 
$\Ihat_{n,\delta}(\cdot)=[\ahat_{n,\delta}(\cdot),\bhat_{n,\delta}(\cdot)]$ consists of a pair of measurable functions $\ahat_{n,\delta},\bhat_{n,\delta}:\R^n\to\R$ with $\ahat_{n,\delta}\leq \bhat_{n,\delta}$. We let $\hat{\ell}_{n,\delta}=\bhat_{n,\delta}-\ahat_{n,\delta}$ denote the length of the interval. We say $\{\Ihat_{n,\delta}\}_{\delta\in[\delta_{\min},1)}$ is a a collection $L$-sub-Gaussian confidence intervals for $(n,\sP,\delta_{\min})$ if 
for any $\Pp\in\sP$, if $X_1^n=_d\Pp^{\otimes n}$, then 
for all $\delta\in[\delta_{\min},1)$,
$$\Pr{\mu_{\Pp}\in \Ihat_{n,\delta}(X_1^n)\mbox{ and }\ellhat_{n,\delta}(X_1^n)\leq
  L\,\sigma_{\Pp}\,\sqrt{\frac{1+\ln(1/\delta)}{n}}}\geq 1-\delta.$$
\end{definition}

The next theorem shows how one can combine sub-Gaussian confidence
intervals to obtain a \mult~sub-Gaussian mean estimator.

\begin{theorem}
\label{thm:confidence}
Let $n$ be a positive integer and let $\sP$ be a class of probability distributions over $\R$.
Assume that there exists a collection of $L$-sub-Gaussian confidence
intervals for $(n,\sP,\delta_{\min})$. 
Then there exists a \mult~estimator $\Ehat_n:\R^n\to\R$ that is $L'$-sub-Gaussian for $(n,\sP,2^{-m})$, where $L'=L\,\sqrt{1+2\ln 2}$ and $m=\lfloor \log_2(1/\delta_{\min})\rfloor - 1\geq \log_2(1/\delta_{\min})-2$ (in particular, $2^{-m}\leq 4\delta_{\min}$).\end{theorem}

\begin{proof}
Our choice of $m$ implies that, for each $k=1,2,3,\dots,m+1$ there exists a measurable closed interval $\Ihat_{k}(\cdot)=[\ahat_k(\cdot),\bhat_k(\cdot)]$ with length $\hat{\ell}_k(\cdot)$, with the property that, if $\Pp\in \sP$ and $X_1^n=_d\Pp^{\otimes n}$, the event
\begin{equation}\label{eq:goodconfidence}\sG_k:=\left\{\mu_{\Pp}\in \Ihat_{k}(X_1^n)\mbox{ and }\hat{\ell}_{k}(X_1^n)\leq L\,\sigma_{\Pp}\,\sqrt{\frac{1+k\ln 2}{n}}\right\}\end{equation}
has probability $\Pr{\sG_k}\geq 1-2^{-k}$. To define our estimator,
define, for $x_1^n\in\R^n$,
$$\hat{k}_n(x_1^n)= \min\left\{k\in[m]\,:\,\bigcap_{j=k}^m\Ihat_j(x_1^n)\neq\emptyset\right\}~.$$
One can easily check that 
$$\bigcap_{j=\hat{k}_n(x_1^n)}^m\Ihat_j(x_1^n)\mbox{ is always a non-empty closed interval,}$$
so it makes sense to define the estimator $\Ehat_n(x_1^n)$ as its
midpoint. 

We claim that $\Ehat_n$ is the sub-Gaussian estimator we are looking for. To prove this, we let $2^{-m}\leq \delta\leq 1$ and choose the smallest $k\in\{1,2,\dots,m+1\}$ with $2^{1-k}\leq \delta$. Assume $X_1^n=_d\Pp^{\otimes n}$ with $\Pp\in\sP$. Then
\begin{enumerate}
\item $\Pr{\bigcap_{j=k}^{m+1}\sG_j}\geq 1 - 2^{-k} - 2^{-k-1} - \dots\geq 1 - 2^{1-k}\geq 1-\delta$ by \eqnref{goodconfidence} and the choice of $k$.
\item When $\bigcap_{j=k}^m\sG_j$ holds, $\mu_{\Pp}\in \Ihat_{j}(X_1^n)$ for all $k\leq j\leq m+1$, so $\mu_{\Pp}\in \bigcap_{j=k}^{m+1}\Ihat_{j}(X_1^n)$. In particular, $\bigcap_{j=k}^m\Ihat_{j}(X_1^n)\neq\emptyset$ and $\hat{k}_n(X_1^n)\leq k$.
\item Now when $\hat{k}_n(X_1^n)\leq k$, $\Ehat_n(X_1^n)\in \bigcap_{j=k}^m\Ihat_{j}(X_1^n)$ as well, so both $\Ehat_n(X_1^n)$ and $\mu_{\Pp}$ belong to $\Ihat_k(X_1^n)$. It follows that $|\Ehat_n(X_1^n)-\mu_{\Pp}|\leq \ellhat_k(X_1^n)$. 
\item Finally, our choice of $k$ implies $2^{1-k}\leq \delta\leq 2^{2-k}$, so, under $\bigcap_{j=k}^m\sG_j$ we have 
$$\hat{\ell}_k(X_1^n)\leq L\,\sigma_{\Pp}\,\sqrt{\frac{1+\ln(2^{k})}{n}}\leq L\,\sigma_{\Pp}\,\sqrt{\frac{1+2\ln 2 + \ln(1/\delta)}{n}} \leq L'\,\sigma_{\Pp}\,\sqrt{\frac{1+\ln(1/\delta)}{n}}.$$ 
with $L '=L\sqrt{1+2\ln 2}$ as in the statement of the theorem. \end{enumerate}
Putting it all together, we conclude
$$\Pr{|\Ehat_n(X_1^n)-\mu_{\Pp}|\leq L'\sigma_{\Pp}\,\sqrt{\frac{1+\ln(1/\delta)}{n}}}\geq \Pr{\bigcap_{j=k}^m\sG_j}\geq 1- \delta,$$
and since this holds for all $\Pp\in\sP$ and all $2^{-m}\leq
\delta\leq 1/2$, the proof is complete.
\end{proof}

\subsection{Scaled Bernoulli distributions and single-$\delta$ estimators}
\label{sec:scaledBer}

In this subsection we prove Theorem \ref{thm:infvar}. In order to do so,
we derive a simple minimax lower bound for
single-$\delta$ estimators for the class $\sP_{c,p}=\{P_+,P_-\}$ of distributions
that contains two discrete distributions defined by
\[
   P_+(\{0\}) =   P_-(\{0\}) = 1-p~,   \qquad  P_+(\{c\}) =    P_-(\{-c\}) = p~,
\]
where $p\in [0,1]$ and $c>0$.
Note that $\mu_{P_+}= pc$, $\mu_{P_-}= -pc$ and that for any
$\alpha>0$, the $(1+\alpha)$-th central moment of both distributions
equals
\begin{equation}
\label{eq:alphamoment}
  M= c^{1+\alpha} p(1-p) \left(p^{\alpha}+(1-p)^{\alpha} \right)~.
\end{equation}

For $i=1,\ldots,n$, let $(X_i,Y_i)$ be independent pairs of
real-valued random variables such that 
\[
   \PROB\{X_i=Y_i=0\} = 1-p\quad \text{and} \quad \PROB\{X_i=c, Y_i=-c\} = p~.
\]
Note that $X_i\haslaw P_+$ and $Y_i\haslaw P_-$. 
Let $\delta \in (0,1/2)$. If  $\delta \ge 2e^{-n/4}$ and  $p =
(2/n)\log(2/\delta)$, then (using $1-p\ge \exp(-p/(1-p))$), 
\[
  \PROB\{X_1^n = Y_1^n\} =(1-p)^n \ge 2\delta~.
\]
Let $\Ehatdelta$ be any mean estimator, possibly depending on $\delta$. Then
\begin{eqnarray*}
\lefteqn{
\max\left(
\PROB\left\{ \left|\Ehatdelta (X_1^n) - \mu_{P_+}\right| >  cp\right\},
\PROB\left\{ \left|\Ehatdelta(Y_1^n) - \mu_{P_-}\right| >  cp \right\} 
\right)
}
\\
& & \ge \frac{1}{2}
\PROB\left\{ \left|\Ehatdelta (X_1^n) - \mu_{P_+}\right| >  cp
  \quad\text{or} \quad
 \left|\Ehatdelta(Y_1^n) - \mu_{P_-}\right| >  cp \right\} 
\\
& &\ge
\frac{1}{2}\PROB\{\Ehatdelta(X_1^n) = \Ehatdelta(Y_1^n)\} 
\\
& &
\ge 
\frac{1}{2} \PROB\{X_1^n = Y_1^n\}\ge \delta~.
\end{eqnarray*}
From (\ref{eq:alphamoment}) we have that $cp\ge M^{1/(1+\alpha)}
(p/2)^{\alpha/(1+\alpha)}$ and therefore
\begin{eqnarray*}
\max\left(
\PROB\left\{ \left|\Ehatdelta (X_1^n) - \mu_{P_+}\right| >
\left(\frac{M^{1/\alpha}}{n}\log\frac{2}{\delta}\right)^{\alpha/(1+\alpha)}\right\},
\right.
\\
\left. \PROB\left\{ \left|\Ehatdelta(Y_1^n) - \mu_{P_-}\right| >  
\left(\frac{M^{1/\alpha}}{n}\log\frac{2}{\delta}\right)^{\alpha/(1+\alpha)} \right\} 
\right) \ge \delta~.
\end{eqnarray*}
 
Theorem \ref{thm:infvar} simply follows by noting that
$\sP_{c,p}\subset \sP^M_{1+\alpha}$.

\subsection{Laplace distributions and single-$\delta$ estimators}
\label{sec:laplace}

This section focuses on the class of {\em all Laplace distibutions
  with scale parameter equal to $1$}. 
To define such a distribution, let $\lambda\in\R$ and let
$\La_\lambda$ be the probability measure on $\R$ with density 
\[
\frac{d\La_{\lambda}}{dx}(x)=\frac{e^{-|x-\lambda|}}{2}~.
\]
Denote by $\sP_{\rm \La}=\{\La_\lambda\,:\, \lambda\in\R\}$ the class of all such distributions. 

A simple calculation reveals that
for all $\lambda\in\R$, the mean, variance, and central third moment
are
$\mu_{\La_{\lambda}}=\lambda$, $\sigma^2_{\La_{\lambda}}=2$ and $\La_{\lambda}|X-\lambda|^{3} = 6\leq (\eta\,\sigma_{\La_{\lambda}})^3$ with $\eta=3^{1/3}\,2^{1/6}$.

The next result proves that \single~$L$-sub-Gaussian estimators are limited to exponentially small $\delta$ even over the one-dimensional family $\sP_{\La}$. 

\begin{theorem}\label{thm:lowerfixed} If $n\geq 3$ then, for any constant $L\geq \sqrt{2}$, there are no \single~$L$-sub-Gaussian estimators for $(\sP_{\La},n,e^{1-5L^2n})$.
\end{theorem}
\begin{proof}We proceed by contradiction, assuming that there exist $L$-sub-Gaussian \single~estimators $\Ehatdelta$ for $(\sP_{\La},n,\delta)$ where $\delta=e^{1-5L^2n}$ and arbitrarily large $n$. We set $$\lambda=2L\sqrt{2\,(1+\ln(1/\delta))/n}$$
and consider $X_1^n=_d\La_0^{\otimes n}$ and
$Y_1^n=_d\La_\lambda^{\otimes n}$. The triangle inequality applied to
the exponents of $d\La_{\lambda}/dx$ and $d\La_{0}/dx$ shows that the
densities of the two product measures satisfy, for all $x_1^n\in\R^n$
\[
\frac{d\La_{0}}{dx_1^n}(x_1^n)\geq
e^{-\eta\,n}\,\frac{d\La_{\lambda}}{dx_1^n}(x_1^n)~,
\]
and therefore,
\begin{equation}
\label{eq:compareexponentials}
\Pr{\Ehatdelta(X_1^n)\geq \frac{\lambda}{2}}\geq e^{-\lambda\,n}\Pr{\Ehatdelta(Y_1^n)\geq \frac{\lambda}{2}}.\end{equation}
Using the definition of $\lambda$ and the fact that 
$\mu_{\La_{\lambda}}=\lambda$ and $\sigma^2_{\La_{\lambda}}=2$, we see
that the right-hand side  above is simply 
$$e^{-\lambda\,n}\,\Pr{\Ehatdelta(Y_1^n)\geq \mu_{\La_\lambda}- L\,\sigma_{\La_\lambda}\sqrt{\frac{1+\ln(1/\delta)}{n}}}\geq e^{-\lambda\,n}\,(1-\delta).$$
On the other hand, the left-hand side in \eqnref{compareexponentials} is
$$\Pr{\Ehatdelta(X_1^n)\geq \mu_{\La_0} + L\sigma_{\La_0}\sqrt{\frac{1+\ln(1/\delta)}{n}}} \leq \delta.$$
We deduce
$$e^{-\lambda\,n}\leq \frac{\delta}{1-\delta}\leq 2\delta.$$
If we use again the definition of $\lambda$, we see that
$$e^{-2L\sqrt{n\,(1+\ln(1/\delta))}}\leq 2\,\delta,$$
or
$$e^{{-}2\sqrt{5}\,L^2\,n}\leq 2e^{1-5L^2n}\Rightarrow n\leq \frac{1 + \ln 2}{L^2\,(5-2\sqrt{5})}~.$$ 
For $L\geq \sqrt{2}$, some simple estimates show that this leads to a contradiction when $n\geq 3$.\end{proof}

\subsection{Poisson distributions and \mult~estimators}
\label{sec:poisson}

We use the family of Poisson distributions for bounding the range of
confidence values of \mult~estimators. Denote by $\Po_\lambda$ the
Poisson distribution with parameter $\lambda >0$.
Given $0< \lambda_1\leq \lambda_2<\infty$, define
$$\sP^{[\lambda_1,\lambda_2]}_{\Po}=\{\Po_{\lambda}\,:\,\lambda\in [\lambda_1,\lambda_2]\}.$$

\begin{theorem}\label{thm:negativestrong}
There exist positive constants $c_0,s_0$ and a function $\phi:\R_+\to\R_+$ such that the following holds. Assume $L\geq \sqrt{2}$ and $n>0$ are given. Then there exists no \mult~$L$-sub-Gaussian estimator for $(\sP^{[c/n,\phi(L)\,c/n]}_{\Po},n,e^{1-s_0\,(L \ln L)^2\,c})$.\end{theorem}

\begin{proof}
We prove the following stronger result: there exist constants
$c_0,s>0$ such that, when $c\geq c_0$, $L\geq \sqrt{2}$ and $C=\lceil
s\,(L^2\ln L)\rceil$, there is no \mult~sub-Gaussian estimator for 
$$(\star) = \left(\sP^{\left[\frac{c}{n},\frac{(1+2C)\,c}{n}\right]}_{\Po},n,e^{1-\frac{C^2c}{L^2}}\right).$$
The theorem then follows by taking $2C=\phi(L)-1=s\,L^2\ln L$ and $s_0=s^2$. 

We proceed by contradiction. Assume $$X_1^n=_d\Po^{\otimes
  n}_{c/n},\;Y_1^n=_d\Po^{\otimes n}_{(1+2C)\,c/n}$$ and that there
exists an $L$-sub-Gaussian estimator $\Ehat_n:\R^n\to\R$ for ($\star$)
above. We use the following well-known facts about Poisson distributions. 
\begin{enumerate}
\item[\bf F0] $\mu_{\Po_{c/n}}=\sigma^2_{\Po_{c/n}}=c/n$ and $\mu_{\Po_{(1+2C)c/n}}=\sigma^2_{\Po_{(1+2C)c/n}}=(1+2C)c/n$.
\item[\bf F1] $S_X=X_1+X_2+\dots+X_n=_d\Po_{c}$ and $S_Y=Y_1+Y_2+\dots+Y_n=_d\Po_{(1+2C)\,c}$.
\item[\bf F2] Given any $k\in\N$, the distribution of $X_1^n$
    conditioned on $S_X=k$ is {\em the same} as the distribution of
    $Y_1^n$ conditioned on $S_Y=k$.
\item[\bf F3]  $\Pr{S_Y=(1+2C)\,c}\geq 1/4\sqrt{(1+2C)\,c}$ if $C>0$
  and $c\geq c_0$ for some $c_0$. 
(This follows from the fact that $\Po_{m}(\{m\}) = e^{-m}m^m/m!$ is asymptotic to $1/\sqrt{2\pi m}$ when $m\to \infty$, by Stirling's formula.)
\item[\bf F4] There exists a function $h$ with $0<h(C)\approx  (1+C)\,\ln (1+C)$ such that, for all $c\geq c_0$, $\Pr{S_X = (1+2C)\,c}\geq e^{-h(C)\,c}$. This follows from another asymptotic estimate proven by Stirling's formula: as $c\to \infty$
$$\Po_{c}(\{(1+2C)\,c\}) = e^{-c}\,\frac{c^{(1+2C)\,c}}{[(1+2C)\,c]!}\sim \frac{e^{-[(1+2C)\,\ln(1+2C) -2 C]\,c}}{\sqrt{2\pi (1+2C)\,c}}~.$$
\end{enumerate}
We apply the sub-Gaussian property for the triple ($\star$) to $\delta=1/4\sqrt{(1+2C)\,c}$. This is possible because, for $C=\lceil s\,(L^2\ln L)\rceil$ with a large enough $s$, this value is $\approx 1/L\sqrt{s\ln L\,c}$, which is much larger than the minimum confidence parameter $e^{1-C^2\,c/L^2}$ allowed by ($\star$) (at least if $c\geq c_0$ with a large enough $c_0$). Recalling {\bf F0}, we obtain
$$\Pr{n\Ehat_n(Y_1^n)<(1+2C)\,c - L\,\sqrt{(1+2C)\,c(1+\ln(8\sqrt{(1+2C)\,c}))}}\leq \frac{1}{4\sqrt{(1+C)\,c}}~.$$ 
Therefore, by {\bf F3},   
$$\Pr{n\Ehat_n(Y_1^n)<(1+2C)\,c - L\,\sqrt{(1+2C)\,c(1+\ln(8\sqrt{(1+2C)\,c}))}\mid S_Y=(1+C)\,c}\leq  1/2.$$
Now {\bf F1} implies that the left-hand side is the same if we switch from $Y$ to $X$. In particular, by looking at the complementary event we obtain
\begin{equation}\label{eq:thisgoes}\Pr{n\Ehat_n(X_1^n)\geq (1+2C)\,c - L\,\sqrt{(1+2C)\,c(1+\ln(8\sqrt{(1+2C)\,c}))}\mid S_X=(1+2C)\,c}\geq 1/2.\end{equation}
Since we are taking $c\geq c_0$ and $C\geq s\,L^2\,\ln L$, a calculation reveals
$$L\sqrt{(1+2C)\,c(1+\ln(8\sqrt{(1+2C)\,c}))} = \bigoh{\sqrt{\frac{C^2\,c\,(\ln C + \ln c)}{\ln C}}} = \bigoh{C\,\sqrt{c\ln c}}.$$
Therefore, by taking a large enough $c_0$ we can ensure that
$$L\sqrt{(1+2C)\,c(1+\ln(8\sqrt{(1+2C)\,c}))}\leq C\,c~.$$
So \eqnref{thisgoes} gives  $$\Pr{n\Ehat_n(X_1^n)\geq (1+C)\,c\mid S_X=(1+2C)\,c}\geq 1/2.$$
We may combine this with {\bf F4} to deduce:
\begin{equation}\label{eq:bound}\Pr{n\Ehat_n(X_1^n)\geq(1+C)\,c}\geq \frac{e^{-h(C)\,c}}{2}.\end{equation}
We now use {\bf F0} to rewrite the previous probability as $$\Pr{n\Ehat_n(X_1^n)\geq  (1+C)\,c}=\Pr{\Ehat_n(X_1^n) - \mu_{\Pp}\geq  L\,\sigma_{\Pp}\frac{\sqrt{1+\ln(1/\delta_0)}}{\sqrt{n}}},$$
where $$\delta_0=e^{1-\frac{C^2\,c}{L^2}}.$$ Since we assumed $\Ehat_n$ is $L$-sub-Gaussian for the triple ($\star$), we obtain
$$\frac{e^{-h(C)\,c}}{2}\leq \Pr{n\Ehat(X_1^n)\geq  (1+C)\,c}\leq e^{1-\frac{C^2\,c}{4L^2}}.$$
Comparing the left and right hand sides, and recalling $c\geq c_0$, we
obtain $h(C)\geq C^2/4L^2-1-(\ln 2/c_0)$. This is a contradiction if
$C\gg L^2\ln L$ because $h(C)$ grows like $C\ln C$ (cf.\ {\bf
  F4}). This contradiction shows that there does not exist a
$L$-sub-Gaussian estimator for $(\star)$, as desired.
\end{proof}

\section{Degrees of knowledge about the variance}\label{sec:variance}

In this section we present the proof of \thmref{secondmoment}. This is mostly a matter of combining the main results in the previous section.
Recall that we consider the class
$$\sP_{2}^{[\sigma_1^2,\sigma_2^2]}=\{\Pp\in\sP_2\,:\, \sigma_1^2\leq \sigma_{\Pp}^2\leq \sigma_2^2.\}$$
and that $R=\sigma_2/\sigma_1$. The three parts of the theorem are proven separately.

\paragraph{Part 1:} (Existence of a \mult~estimator with constant
depending on $R$.) \thmref{MoM2} ensures that, irrespective of
$\sigma_1$ or $\sigma_2$, for all $\delta \in (e^{1-n/2},1)$ 
there exists a \single~estimator $\Ehat_{n,\delta}:\R^n\to\R$ with
\begin{equation}\label{eq:almostconfidence}
\Pr{|\Ehat_{n,\delta}(X_1^n)-\mu_{\Pp}|>2\sqrt{2}{e}\,\sigma_{\Pp}\,\sqrt{\frac{1+\ln(1/\delta)}{n}}}\leq
  \delta
\end{equation}
whenever $X_1^n=\Pp^{\otimes n}$ for some $\Pp\in\sP_{{\rm 2}}$. We define a confidence interval for each $\delta$ via 
$$\Ihat_{n,\delta}(x_1^n)=\left[\Ehat_{n,k}(x_1^n)-2\sqrt{2}\,e\,\sigma_2\,\sqrt{\frac{1+\ln(1/\delta)}{n}},\Ehat_{n,k}(x_1^n)+2\sqrt{2}\,e\,\sigma_2\,\sqrt{\frac{1+\ln(1/\delta)}{n}}\right].$$
Clearly, \eqnref{almostconfidence} and the fact that $\sigma_{2}\leq R\,\sigma_{\Pp}$ for all $\sP_{2}^{[\sigma_1^2,\sigma_2^2]}$ imply that $\{\Ihat_{n,\delta}\}_{\delta\in[e^{1-n/2},1)}$ is a $4\sqrt{2}\,{e}\,R$-sub-Gaussian confidence interval for $(\sP^{[\sigma_1^2,\sigma_2^2]}_{ 2},n,e^{1-n/2})$. Applying \thmref{confidence} gives the desired result. 

\paragraph{Part 2:}
 (Non-existence of \mult~estimators when $R>\phi^{(2)}(L)$.) We use \thmref{negativestrong}. By rescaling, we may assume $\sigma_1^2=c_0/n$, where $c_0$ is the constant appearing in \thmref{negativestrong}. We also set $\phi^{(2)}(L):=\sqrt{\phi(L)}$ for $\phi(L)$ as in \thmref{negativestrong}. The assumption on $R$ ensures that $\sP^{[c_0/n,\phi(L)\,c_0/n]}_{\Po}\subset \sP_{2}^{[\sigma_1^2,\sigma_2^2]}$, so there cannot be a $L$-sub-Gaussian estimator when $\delta^{(2)}_{\min}(L)=e^{1-s\,(L\ln L)^2\,c_0}$.

\paragraph{Part 3:} (Non-existence of \single~estimators when $\delta_{\min}= e^{1-5L^2n}$.) By rescaling, we may assume $\sigma_1^2=2$. Then the class $\sP_{\La}$ in \thmref{lowerfixed} is contained in $\sP_{2}^{[\sigma_1^2,\sigma_2^2]}$, and the theorem implies the desired result directly.

\section{The regularity condition, symmetry and higher moments}\label{sec:regular_proofs}

In this section we prove the results described in \secref{regularresults}. 

\subsection{An estimator under $k$-regularity}\label{sec:regular_main}

We start with \thmref{regular}, the general positive result on $k$-regular classes.

\medskip
\noindent
{\sl Proof of \thmref{regular}:} 
By \thmref{confidence}, it suffices to build a $4\sqrt{2\,(1+62\ln(3))}\,e^{\frac{5}{2}}$-sub-Gaussian confidence interval for $(\sP_{2,k{\rm -reg}},n,e^{3-n/(124)k})$. 

To build these intervals, we use an idea related to the proof of
\thmref{MoM2}. Just like in the case of the median-of-means estimator, we divide the data into blocks, but instead of taking the median of the means, we look at the $1/4$ and $3/4$-quantiles to build an interval. 

To make this precise, given $\alpha\in (0,1)$, we define the $\alpha$-quantile $q_\alpha(y_1^b)$ of a vector $y_1^b\in\R^b$ as the smallest index $i\in[b]$ with
$${\# \{j\in[b]\,:\,y_j\leq y_i\}\geq \alpha\,b \quad \mbox{ and }}
\quad {\# \{\ell\in[b]\,:\,y_\ell\geq y_i\}\geq (1-\alpha)\,b}.$$

The next result (proven subsequently) is an analogue of \lemref{median}.

\begin{lemma}\label{lem:quantile} 
Let $Y_1^b=(Y_1,\dots,Y_b)\in\R^b$ be a vector of independent random variables with the same mean $\mu$ and variances bounded by $\sigma^2$. Assume further that $\Pr{Y_i\leq \mu}\geq 1/3$ and $\Pr{Y_i\geq \mu}\geq 1/3$ for each $i\in[b]$. Then
$$\Pr{\mu\in [q_{1/4}(Y_1^b),q_{3/4}(Y_1^b)]\mbox{ and }q_{3/4}(Y_1^b)-q_{1/4}(Y_1^b)\leq 2L_0\,\sigma}\geq 1 - 3\,e^{-db},$$
where $d$ is the numerical constant
$$d=\frac{1}{4}\,\ln\left(\frac{3}{4}\right) +
\frac{3}{4}\,\ln\left(\frac{9}{8}\right)
\approx 0.0164 >
\frac{1}{62}$$ and
$L_0=2\,e^{2d+\frac{1}{2}}\leq 2\,e^{\frac{5}{2}}.$\end{lemma}

Now fix $\delta\in [e^{3-n/(124 k)},1)$. We define a confidence interval $\Ihat_{n,\delta}(\cdot)$ as follows. First set $b=\lceil 62\ln(3/\delta)\rceil$ and note that
\begin{equation}\label{eq:bdelicate}b\leq 62\ln(3/e^{3-n/(124k)})+1\leq \frac{n}{2k}\leq n/2.\end{equation}
Partition $$[n] = B_1\cup B_2\cup \dots\cup B_b$$into disjoint blocks
of sizes $|B_i|\geq \lfloor n/b\rfloor$. For each $i\in[b]$ and $x_1^n\in\R^n$, we define
$$y_1^b(x_1^n) = (y_1(x_1^n),\dots,y_b(x_1^n))\mbox{ where }y_i(x_1^n)=\frac{1}{\# B_i}\sum_{j\in B_i}x_j$$
and set, for $x_1^n\in\R^n$,
$$\Ihat_{n,\delta}(x_1^n)=\left[q_{1/4}(y_1^b(x_1^n)),q_{3/4}(y_1^b(x_1^n))\right]~.$$

\begin{claim}$\{\Ihat_{n,\delta}(\cdot)\}_{\delta\in [e^{3-n/(124k)},1)}$ is a $4\sqrt{2\,(1+62\ln(3))}\,e^{\frac{5}{2}}$-sub-Gaussian collection of confidence intervals for $(\sP_{2,k{\rm -reg}},n,e^{3-n/(124 k)})$.\end{claim}

To see this, we take a distribution $\Pp$ in this family and assume $X_1^n=_d\Pp^{\otimes n}$. Set $s=\lfloor n/b\rfloor$. Because the blocks $B_i$ are disjoint and have at least $s$ elements each, the random variables 
$$Y_i=y_i(X_1^n) = \Pphat_{B_i}X,$$
all have mean $\mu_{\Pp}$ and variance $\leq \sigma_{\Pp}^2/s$. Moreover, using \eqnref{bdelicate}, 
$$s=\left\lfloor \frac{n}{b}\right\rfloor \geq \frac{n}{b}-1 \geq 2k-1\geq k,$$
so the $k$-regularity property implies  that for all $i\in [b]$,
$$\Pr{Y_i\leq \mu}\geq \frac{1}{3},\,\Pr{Y_i\geq \mu}\geq \frac{1}{3}.$$
\lemref{quantile} implies
\begin{equation}\label{eq:claimfollows}\Pr{\mu_{\Pp}\in \Ihat_{n,\delta}(X_1^n)\mbox{ and length of }\Ihat_{n,\delta}(X_1^n)\leq 2\,L_0\,\frac{\sigma}{\sqrt{s}}}\geq 1 - 3\,e^{-db}\geq 1-\delta\end{equation}
by the choice of $b$ and the fact that $d\geq 1/62$. To finish, we use \eqnref{bdelicate} and the definition of $b$ to obtain$$\frac{1}{s}=\frac{1}{\lfloor n/b\rfloor}\leq  \frac{1}{(n/b)-1}\leq \frac{2b}{n} \leq \frac{2(\lceil 62\ln(3) +62\ln(1/\delta)\rceil)}{n}\leq 2(1+62\ln 3)\,\frac{1+\ln(1/\delta)}{n}.$$ Plugging this back into \eqnref{claimfollows} and recalling $L_0\leq 2e^{5/2}$ implies the desired result.

\medskip
\noindent
{\sl Proof of \lemref{quantile}:} 
Define $J=[\mu-L_0\sigma,\mu+L_0\sigma]$. Assume the following three properties hold.
\begin{enumerate}
\item $q_{1/4}(Y_1^b)\leq \mu$.
\item $q_{3/4}(Y_1^b)\geq \mu$. 
\item The number of indices $i\in[b]$ with $Y_i\in J$ is at least $3b/4$.
\end{enumerate}
Then clearly $\mu\in [q_{1/4}(Y_1^b),q_{3/4}(Y_1^b)]$. Moreover, item 3 implies that $q_{1/4}(Y_1^b),q_{3/4}(Y_1^b)\in J$, so that
$$q_{3/4}(Y_1^b)-q_{1/4}(Y_1^b)\leq\mbox{ (length of $J$) }= 2L_0\,\sigma.$$
It follows that
\begin{multline}\label{eq:threeterms}\Pr{\mu\not\in  [q_{1/4}(Y_1^b),q_{3/4}(Y_1^b)]\mbox{ or }q_{3/4}(Y_1^b)-q_{1/4}(Y_1^b)>2L_0\,\sigma}\\ \leq \Pr{q_{1/4}(Y_1^b)> \mu} + \Pr{q_{3/4}(Y_1^b)< \mu} + \Pr{\#\{i\in [b]\,:\, Y_i\not \in J\}>b/4}.\end{multline}
We bound the three terms by $e^{-bd}$ separately. By assumption, $\Pr{Y_i\leq \mu}\geq 1/3$ for each $i\in[b]$. Since there events are also independent, we have that 
$\sum_{i=1}^b\Ind{Y_i\leq \mu}$ stochastically dominates a binomial
random variable ${\rm Bin}(b,1/3).$
Thus, 
$$\Pr{q_{1/4}(Y_1^b)>\mu} = \Pr{\sum_{i=1}^b\Ind{Y_i\leq \mu}<b/4}\leq \Pr{{\rm Bin}(b,1/3)<b/4}\leq e^{-db}$$
by the relative entropy version of the Chernoff bound and the fact
that $d$ is the relative entropy between two Bernoulli distributions with parameters $1/4$ and $1/3$. A similar reasoning shows that $\Pr{q_{3/4}(Y_1^b)> \mu} \leq e^{-db}$ as well.

It remains to bound $\Pr{\#\{i\in [b]\,:\, Y_i\not\in J\}>b/4}$. To this end note that
for all $i\in[b]$,
\begin{equation}\label{eq:notinJbound}\Pr{Y_i\not\in J} = \Pr{|Y_i-\mu|\geq L_0\sigma}\leq \frac{1}{L_0^2},\end{equation}
and these events are independent. It follows that  
\begin{eqnarray*}\Pr{\#\{i\in [b]\,:\, Y_i\not\in J\}>b/4}&\leq & \Pr{\bigcup\limits_{A\subset [b],\, |A|=\lceil b/4\rceil}\bigcap_{i\in A}\{Y_i\not\in J\}}\\ 
\mbox{(union bound)}&\leq & \binom{b}{\lceil b/4\rceil}\,\max\limits_{A\subset [b],\, |A|=\lceil b/4\rceil}\Pr{\bigcap_{i\in A}\{Y_i\not\in J\}}\\
\mbox{(independence of $Y_i$ +\eqnref{notinJbound})} &\leq &\binom{b}{\lceil b/4\rceil}\,\left(\frac{1}{L_0^2}\right)^{-\left\lceil \frac{b}{4}\right\rceil}\\
\mbox{($\binom{b}{k}\leq (eb/k)^{k}$ for all $1\leq k\leq b$)} &\leq &\left(\frac{e\,b}{L^2_0\,\lceil b/4\rceil}\right)^{\left\lceil\frac{b}{4}\right\rceil}\\ \mbox{($b\leq 4\lceil b/4\rceil$ and $L_0^2=4e^{4d+1}$)} &\leq & e^{-4d\left\lceil \frac{b}{4}\right\rceil}\leq e^{-bd}.\end{eqnarray*}

\subsection{Symmetric distributions}\label{sec:symmetric}

To prove \thmref{symmetric}, notice that the existence of the \mult~sub-Gaussian estimator follows from \thmref{regular}. The second part is a simple consequence of \thmref{lowerfixed} and the fact that Laplace distributions are symmetric around their means.

\subsection{Higher moments}\label{sec:alphaclass}

In this section we first prove that $\sP_{\alpha,\eta}\subset
\sP_{2,k{\rm -reg}}$ for large enough $k$, and then prove
\thmref{alphaclass}. We recall the definition of $\min(p_+(\Pp,j)$ and
$p_-(\Pp,j)$ from \defref{regular}.

\begin{lemma}\label{lem:kregmoments}
For all $\alpha\in (2,3]$, there exists $C=C_\alpha$ such that, if
$j\geq (C_{\alpha}\eta)^{\frac{2\alpha}{\alpha-2}}$, then
$\min(p_+(\Pp,j),p_-(\Pp,j)) \geq 1/3$. 
\end{lemma}

\begin{proof}
We only prove that $p_+(\Pp,j)\geq 1/3$, as the other proof is analogous. 

Let $N$ be a standard normal random variable. Take some smooth
function $\Psi:\R\to\R$ with bounded second and third derivatives,
such that $\Psi(x)=0$ for $x\in (-\infty,0]$, $0\leq \Psi(x)\leq 1$
for $x>0$ and $\Ex{\Psi(N)}\geq 1/\sqrt{6}$. (It is easy to see that
such a $\Psi$ exists.) 
Also let $X_1^j=_d\Pp^{\otimes j}$ and assume, without loss of
generality, that $\sigma_{\Pp}>0$. Then
$$p_+(\Pp,j) = \Pr{\frac{1}{\sigma^2_\Pp\sqrt{j}}\sum_{i=1}^j(X_i-\mu_{\Pp})\geq 0}\geq \Ex{\Psi\left(\frac{1}{\sigma^2_\Pp\sqrt{j}}\sum_{i=1}^j(X_i-\mu_{\Pp})\right)}.$$
Lindberg's proof of the central limit theorem (see \cite{Stroock_ProbTheory}), specialized to the case where $X_1^j$ are i.i.d., gives  
$$\Ex{\Psi\left(\frac{1}{\sigma^2_\Pp\sqrt{j}}\sum_{i=1}^j(X_i-\mu_{\Pp})\right)}\geq \Ex{\Psi(N)} - C_0\,j\,\Pp\left[{\phi\left(\frac{X-\mu_{\Pp}}{\sigma_{\Pp}\sqrt{j}}\right)}\right],$$
where $\phi(t) = t^2\wedge t^3$ and $C_0>0$ is a universal constant. 
Since $\Ex{\Psi(N)}\geq 1/\sqrt{6}>1/3$ and $\phi(t)\leq t^{\alpha}$, we obtain
$$\Ex{\Psi\left(\frac{1}{\sigma^2_\Pp\sqrt{j}}\sum_{i=1}^j(X_i-\mu_{\Pp})\right)}\geq \frac{1}{\sqrt{6}} - C_0\,j^{\frac{\alpha}{2}-1}\,\eta^\alpha~.$$
The right-hand side is $\geq 1/3$ when $j\geq (C\eta)^{\frac{2\alpha}{\alpha-2}}$ for some universal $C=C_\alpha$. \end{proof}

\medskip
\noindent
{\sl Proof of \thmref{alphaclass}:} 
The positive result follows directly from \thmref{regular} plus
\lemref{kregmoments}, which guarantees $p_\pm(\Pp,j)\geq 1/3$ for
$j\geq  k_{\alpha}$. For the second part, we first assume
$\eta>\eta_0$ for a sufficiently large constant $\eta_0$. We use the Poisson family of distributions from \secref{poisson}. For $\lambda=\liloh{1}$ and $\alpha\in(2,3]$, we have that
$$\Po_{\lambda}|X-\lambda|^{\alpha} = (1+\liloh{1})\,\lambda = (1+\liloh{1})\,\sigma_{\Po_{\lambda}}^{\alpha}\,\lambda^{-\frac{\alpha-2}{2}}.$$
If we compare this to Example \ref{example:alphaclass}, we see that $\Po_{\lambda}\in \sP_{\alpha,\eta}$ if $\lambda \geq h\,/\eta^{2\alpha/(\alpha-2)}$ for some constant $h=h_\alpha>0$ (recall we are assuming that $\eta\geq \eta_0$ is at least a large constant). Now take $c>0$ such that $c/n=h\,/\eta^{2\alpha/(\alpha-2)}$. If $c>c_0$ for the constant $c_0$ in the statement of \thmref{negativestrong}, we can apply the theorem to deduce that there is no \mult~estimator for $(\sP^{[c/n,\phi(L)\,c/n]}_{\Po},n,e^{-\,c})$. Noting that $c$ is of the order $n/k_\alpha$ finishes the proof in this case. 

Now assume $\eta\leq \eta_0$. In this case we use the Laplace
distributions in \secref{laplace}. Since $2<\alpha\leq 3$, we may
apply the fact that the central third moment of a Laplace distribution satisfies
$\La_{\lambda}|X-\lambda|^{3} = 6\leq (3^{1/3}\,2^{1/6}\sigma_{\La_{\lambda}})^3$ 
to obtain
$$\La_{\lambda}|X-\lambda|^{\alpha}\leq \left(\La_{\lambda}|X-\lambda|^{3}\right)^{\alpha/3}\leq (3^{1/3}2^{1/6}\,\sigma_{\La_{\lambda}})^\alpha.$$
Our assumption on $\eta$ implies that $\sP_{\La}\subset \sP_{\alpha,\eta}$. Thus \thmref{lowerfixed} implies that there is no \single~or \mult~sub-Gaussian estimator for $(\sP_{\alpha,\eta},n.e^{1-5L^2n})$. This is the desired result since $k_\alpha$ is bounded when $\eta\leq \eta_0$. 

Finally, the third part of the theorem follows from the same reasoning as in the previous paragraph.

\section{Bounded kurtosis and nearly optimal constants}\label{sec:BoundedKurtosis}

In this section we prove \thmref{OptConstUnderKurtosis}. Throughout the proof we assume $X=_d \Pp$ and $X_1^n=_d\Pp^{\otimes n}$ for some $\Pp\in\sP_{{\rm krt}\leq \kappa}$, and let $b_{\max}$, $C$, $\xi$ be as in \secref{kurtosis}. Our proof is divided into four steps. 

\begin{enumerate}
\item {\em Preliminary estimates for mean and variance.} 
We use the median-of-means technology to obtain preliminary estimates
for the mean and variance of $\Pp$. These estimates are not good
enough to satisfy the claimed properties, but with extremely high probability they are reasonably close to the true values.
\item {\em Truncation at the ideal point.} We introduce a two-parameter family of truncation-based estimators for $\mu_{\Pp}$, and analyze the behavior of one such estimator, chosen under knowledge of $\mu_{\Pp}$ and $\sigma_{\Pp}$.
\item {\em Truncated estimators are  insensitive.} Finally, we use a chaining argument to show that this two-parameter family is insensitive to the choice of parameters. 
\item {\em Wrap up.} The insensitivity property means that the preliminary estimates from Step 1 are good enough to ``make everything work."
\end{enumerate} 

We conclude the section by a remark on how to obtain a broader range of
$\delta_{\min}$ with a worse constant $L$.

\paragraph{Step 1.}  (Preliminary estimates via median of means.)
Denote by $\muhat_{b_{\max}}=\muhat_{b_{\max}}(X_1^n)$ the estimator given by \thmref{MoM2} with $\delta=e^{-b_{\max}}$, which is possible if $C\ge 6/(1-\log 2)$. The next lemma provides an estimator of the variance.
\begin{lemma}\label{lem:MoMVariance}
Let $B_1,\ldots, B_{b_{\max}}$ denote a partition of $[n]$ into blocks of size $|B_i|\ge k=\lfloor n/b_{\max}\rfloor\ge 2$. For each block $B_i$ with $i\in[b_{\max}]$, define 
\[\shat_{i}^2=\frac1{|B_i|(|B_i|-1)}\sum_{j\ne k\in B_i}(X_j-X_k)^2\quad \mbox{ and }\quad\nuhat_{b_{\max}}^2=q_{1/2}(\shat_1,\ldots,\shat_{b_{\max}})~.\]
Then
\[\Pr{\absj{\nuhat_{b_{\max}}^2-\sigma_{\Pp}^2}\le 2e\sqrt{6(\kappa+3)}\sigma_{\Pp}^2\sqrt{\frac{b_{\max}}{n}}}\ge 1-e^{-b_{\max}}~.\]
In particular, if 
\[\frac{96 e (\kappa+3)\,b_{\max}}n\le 1~,\]
then
\[\Pr{\absj{\muhat_{b_{\max}}-\mu_{\Pp}}\le
  2\sqrt{2}\,e\nuhat_{b_{\max}}\sqrt{\frac{b_{\max}}n} \ \mbox{ and }
  \ \nuhat_{b_{\max}}^2\le \frac32\sigma_{\Pp}^2}\ge 1-2e^{-b_{\max}}~.\]
\end{lemma}
\begin{proof}
 Compute
 
\begin{align*}
 \Ex{\shat_{i}^4}&=\frac1{|B_i|^2(|B_i|-1)^2}\sum_{(j, k)\in B_i^{(2)}}\Ex{(X_j-X_k)^4}\\
 &+\frac6{|B_i|^2(|B_i|-1)^2}\sum_{(j,k,l)\in B_i^{(3)}}\Ex{(X_j-X_k)^2(X_j-X_l)^2}\\
&+\frac1{|B_i|^2(|B_i|-1)^2}\sum_{(j,k,l,m)\in B_i^{(4)}}\Ex{(X_j-X_k)^2(X_l-X_m)^2}
\end{align*}
Expanding all the squares, using independence and noticing that $\Ex{X_j-\mu_{\Pp}}=0$, we get
\[\Ex{(X_j-X_k)^4}=2(\kappa_{\Pp}+3)\sigma_{\Pp}^4~,\]
\[\Ex{(X_j-X_k)^2(X_j-X_l)^2}=(\kappa_{\Pp}+3)\sigma_{\Pp}^4,\; \Ex{(X_j-X_k)^2(X_l-X_m)^2}=4\sigma_{\Pp}^4~. \]
Therefore,
\[ \Ex{\shat_{i}^4}\le \paren{\frac{3(\kappa_{\Pp}+3)}{|B_i|}+1}\sigma_{\Pp}^4\le \Ex{\shat_i^2}^2+6(\kappa+3)\sigma_{\Pp}^4\frac{b_{\max}}{n}~.\]
\lemref{median} with $L_0=e$ gives then
\[\Pr{\absj{\nuhat_{b_{\max}}^2-\sigma_{\Pp}^2}>2e\sqrt{6(\kappa+3)}\sigma_{\Pp}^2\sqrt{\frac{b_{\max}}{n}}}\le e^{-b_{\max}}~.\]
In particular, we get
\[\Pr{\frac12\sigma_{\Pp}^2\le \nuhat_{b_{\max}}^2\le
  \frac32\sigma_{\Pp}^2}\ge 1- e^{-b_{\max}}~.\]
The theorem follows by the definition of $\muhat_{b_{\max}}$ and an application of \thmref{MoM2}.
\end{proof}

\paragraph{Step 2:}  (Two-parameter family of estimators at the ideal
point.) 
Given $\mu$ and $R$ define,  for all $x\in \R$,
\[
\Psi_{\mu,R}(x)=\mu+\paren{\frac{R}{\absj{x-\mu}}\wedge
  1}\paren{x-\mu}~.
\]
\begin{lemma}\label{lem:Benett}
Assume $b_{\max}\ge t$, $R=\sigma_{\Pp}\sqrt{n/b_{\max}}$ and
$\mu=\mu_{\Pp}$. Then, with probability at least $1-2e^{-t}$,
\[\absj{\Pphat_n\Psi_{\mu,R}-\mu_{\Pp}}\le 2\sqrt2\kappa_{\Pp}\sigma_{\Pp} \paren{\frac{b_{\max}}{n}}^{3/2}+\sqrt{\frac{2t}{n}}\sigma_{\Pp}\paren{1+\frac{1}{3\sqrt{2}}\sqrt{\frac{\kappa_{\Pp} t}n}+\frac5{48}\frac{\kappa_{\Pp}t}n}§~.\] 
\end{lemma}
\begin{proof}
 The proof is a consequence of Benett's inequality. It suffices to estimate the moments of $\Psi_{\mu,R}(X)-\mu_{\Pp}$.
 For the first moment, 
\begin{align*}
 \absj{\Ex{\Psi_{\mu,R}(X)-\mu_{\Pp}}}&=\absj{\Ex{\paren{1\wedge\frac{R}{\absj{X-\mu_{\Pp}}}-1}\paren{X-\mu_{\Pp}}}}\\
 &\le\Ex{\paren{1-\frac{R}{\absj{X-\mu_{\Pp}}}}_+\absj{X-\mu_{\Pp}}} \\
 &\le \Ex{\absj{X-\mu_{\Pp}}\Ind{\absj{X-\mu_{\Pp}}>R}} \\
 &\le\Ex{\absj{X-\mu_{\Pp}}^4}^{1/4}\Pr{\absj{X-\mu_{\Pp}}>R}^{3/4} \le\frac{\kappa_{\Pp} \sigma_{\Pp}^4}{R^3} 
\end{align*}
where we used H\"older's inequality. On the other hand,
\begin{align*}
\Ex{\paren{\Psi_{\mu,R}(X)-\mu_{\Pp}}^2}\le \sigma_{\Pp}^2
 \end{align*}
 By the Cauchy-Schwarz inequality, and using the bounded kurtosis assumption,
 \begin{align*}
\Ex{ \absj{\Psi_{\mu,R}(X)-\mu_{\Pp}}^3}&=\Ex{\absj{1\wedge\frac{R}{\absj{X-\mu_{\Pp}}}}^3\absj{X-\mu_{\Pp}}^3}\le \Ex{\absj{X-\mu_{\Pp}}^3}\le \sqrt{\kappa_{\Pp}}\sigma_{\Pp}^3~.
 \end{align*}
 Finally, for any $p\ge 4$, since $\absj{\Psi_{\mu,R}(X)-\mu_{\Pp}}^p\le R$,
 \[\Ex{\absj{\Psi_{\mu,R}(X)-\mu_{\Pp}}^p}\le R^{p-4}\kappa_{\Pp}\sigma_{\Pp}^4~.\]
 For $s=\sqrt{2nt}/\sigma_{\Pp}$, we have $\absj{sR}/n\le 1$ and therefore
\begin{eqnarray*}
\lefteqn{ \Ex{e^{\frac{s}n\paren{\Psi_{\mu,R}(X)-\mu_{\Pp}}}}   } \\
&\le & 1+\frac sn\frac{\kappa_{\Pp} \sigma_{\Pp}^4}{R^3}+\frac{s^2}{2n^2}\sigma_{\Pp}^2+\frac{s^3}{6n^3}\sqrt{\kappa_{\Pp}}\sigma_{\Pp}^3+\frac{s^4}{24n^4}\kappa_{\Pp}\sigma_{\Pp}^4\paren{1+\sum_{p\ge 5}\frac{4!}{p!}}\\
 &\le & \exp\paren{2\sqrt2\frac sn\kappa_{\Pp}\sigma_{\Pp} \paren{\frac{b_{\max}}{n}}^{3/2}+\frac{s^2}{2n^2}\sigma_{\Pp}^2+\frac{s^3}{6n^3}\sqrt{\kappa_{\Pp}}\sigma_{\Pp}^3+\frac{5s^4}{96n^4}\kappa_{\Pp}\sigma_{\Pp}^4}~.
\end{eqnarray*}
By Chernoff's bound,
\[\Pr{\Pphat_n\Psi_{\mu,R}-\mu_{\Pp}>2\sqrt2\kappa_{\Pp}\sigma_{\Pp} \paren{\frac{b_{\max}}{n}}^{3/2}+\frac{s}{n}\sigma_{\Pp}^2+\frac{s^2}{6n^2}\sqrt{\kappa_{\Pp}}\sigma_{\Pp}^3+\frac{5s^3}{96n^3}\kappa_{\Pp}\sigma_{\Pp}^4}\le e^{-\frac{s^2\sigma_{\Pp}^2}{2n}}\]
or, equivalently,
\[\Pr{\Pphat_n\Psi_{\mu,R}-\mu_{\Pp}>2\sqrt2\kappa_{\Pp}\sigma_{\Pp} \paren{\frac{b_{\max}}{n}}^{3/2}+\sqrt{\frac{2t}{n}}\sigma_{\Pp}\paren{1+\frac{1}{3\sqrt{2}}\sqrt{\frac{\kappa_{\Pp} t}n}+\frac5{48}\frac{\kappa_{\Pp}t}n}}\le e^{-t}~.\]
Repeat the same computations with $s=-\sqrt{2nt}/\sigma_{\Pp}$ to prove the lower bound.
\end{proof}

\paragraph{Step 3:} (Insensitivity of the estimators.) 
Given $\epsilon_\mu,\epsilon_R \in (0,1/2)$, define
\[\sR=\set{(\mu,R)\;:\; \absj{\mu-\mu_{\Pp}}\le \epsilon_\mu\sigma_{\Pp},\quad \absj{R-\sigma_{\Pp}\sqrt{n/(2b_{\max})}}\le \epsilon_R\sigma_{\Pp}}~,\]
\[\Delta_{\mu,R}=\Pphat_n\paren{\Psi_{\mu,R}-\Psi_{\mu_{\Pp},\sigma_{\Pp}\sqrt{n/(2b_{\max})}}}~.\]
\begin{lemma}
\label{lem:ConcDelta}
Assume $\sqrt{n/(2b_{\max})}\ge 2(\epsilon_\mu+\epsilon_R)$ then for
any $t>0$, with probability at least $1-e^{-t}$, for all $(\mu,R)\in \sR$,
\[\absj{\Delta_{\mu,R}}\le
(\epsilon_\mu+\epsilon_R)\sigma_{\Pp}\paren{\frac{56b_{\max}}n+\frac{4\sqrt{b_{\max}t}}n+\frac{2t}{3n}}~.
\] 
\end{lemma}
\begin{proof}
 Start with the trivial bound
 \[
\absj{\Psi_{\mu,R}(x)-\Psi_{\mu',R'}(x)}\le
\absj{\mu-\mu'}+\absj{R-R'}
\]
that holds for all $(\mu,R),(\mu',R')\in\sR$ and $x\in\R$.
 Moreover, assume that $\absj{x-\mu_{\Pp}}\le
 \sigma_{\Pp}\sqrt{n/(8b_{\max})}$. Then 
 \[\absj{x-\mu}\le \absj{x-\mu_{P}}+\epsilon_\mu\sigma_{\Pp}\le \sigma_{\Pp}\paren{2\sqrt{n/(8b_{\max})}-\epsilon_R}\le R~.\]
 Hence, for any $x\in \R$ such that $\absj{x-\mu_{\Pp}}\le \sigma_{\Pp}\sqrt{n/(8b_{\max})}$
and for all $(\mu,R)\in\sR$,
\[
\Psi_{\mu,R}(x)=x~.
\]
Therefore, for any $(\mu,R)$ and $(\mu',R')$ in $\sR$ and for any $x\in\R$,
 \[ \absj{\Psi_{\mu,R}(x)-\Psi_{\mu',R'}(x)}\le \paren{\absj{\mu-\mu'}+\absj{R-R'}}\Ind{\absj{x-\mu_{\Pp}}> \sigma_{\Pp}\sqrt{n/(8b_{\max})}}~.\]
By Chebyshev's inequality, this implies that, for any positive
integer $p$,
\[\Pp\absj{\Psi_{\mu,R}-\Psi_{\mu',R'}}^p\le \paren{\absj{\mu-\mu'}+\absj{R-R'}}^p\frac{8b_{\max}}n~.\]
By Bennett's inequality, 
\[\Pr{\frac{\absj{\Delta_{\mu,R}-\Delta_{\mu',R'}}}{\absj{\mu-\mu'}+\absj{R-R'}}>\frac{8b_{\max}}n+\frac{4\sqrt{b_{\max}t}}n+\frac{t}{3n}}\le 2e^{-t}~.\]
To apply a chaining argument, consider the sequence $(D_j)_{j\ge 0}$
of points of $\sR$ obtained by the following
construction. $D_0=(\mu_{\Pp},\sigma_{\Pp}\sqrt{n/(2b_{\max})}$ and,
for any $j\ge 1$, divide $\sR$ into $4^j$ pieces by dividing each axis
into $2^j$ pieces of equal sizes. Define then $D_j$ as the set of
lower left corners of the $4^j$ rectangles. Then $|D_j|= 4^j$ and, for any $(\mu,R)\in \sR$, there exists a point $\pi_j(\mu,R)\in D_j$ such that the $\ell_1$-distance between $(\mu,R)$ and $\pi_j(\mu,R)$ is upper-bounded by $2^{-j}(\epsilon_\mu+\epsilon_R)\sigma_{\Pp}$. Therefore, 
\[\sup_{(\mu,R)\in \sR}\absj{\Delta_{(\mu,R)}}\le \sum_{j\ge 1}\sup_{(\mu,R)\in D_j}\absj{\Delta_{(\mu,R)}-\Delta_{\pi_{j-1}(\mu,R)}}~.\]
A union bound in Bennett's inequality gives that, with probability at least $1-2^{1-j}e^{-t}$, for any $(\mu,R)\in D_j$,
\[\absj{\Delta_{\mu,R}-\Delta_{\pi_{j-1}(\mu,R)}}\le (\epsilon_\mu+\epsilon_R)\sigma_{\Pp}\paren{\frac{16b_{\max}}{2^{j}n}+\frac{8\sqrt{b_{\max}(t+j\log 8)}}{2^jn}+\frac{2t+2j\log 8}{3n2^j}}~.\]
Summing up these inequalities gives the desired bound.
\end{proof}
\begin{corollary}\label{cor:AlmostThere}
 Assume $t\ge 1$, $\sqrt{n/(2b_{\max})}\ge
 2(\epsilon_\mu+\epsilon_R)$. Then, with probability at least $1-2e^{-t}-2e^{-b_{\max}}$,
for all $(\mu,R)\in\sR$,
\[
\absj{\Pphat_n\Psi_{\mu,R}-\mu_{\Pp}}\le \frac{\sigma_{\Pp}}{\sqrt{n}}\paren{\sqrt{2t}(1+\xi_1)+\xi_2}~,\]
 where
 \[\xi_1=\frac{1}{3\sqrt{2}}\sqrt{\frac{\kappa_{\Pp} t}n}+\frac5{48}\frac{\kappa_{\Pp}t}n+2(\epsilon_\mu+\epsilon_R)\paren{\sqrt{\frac{2b_{\max}}n}+\frac13\sqrt{\frac{ t}n}}~,\]
 \[\xi_2=2\sqrt2\kappa_{\Pp}\frac{b_{\max}^{3/2}}{n}+ 56(\epsilon_\mu+\epsilon_R)\frac{b_{\max}}n~.\]
\end{corollary}

\paragraph{Step 4:} (Wrap-up.)  Define now 
\[(\muhat_n,\Rhat_n)=\paren{\muhat_{b_{\max}},\nuhat_{b_{\max}}\sqrt{\frac n{2b_{\max}}}}\]
From \lemref{MoMVariance}, with probability at least $1-2e^{-b_{\max}}$,
\[\absj{\muhat_n-\mu}\le 2\sqrt{2}\,e\sigma_{\Pp}\sqrt{\frac{b_{\max}}n},\quad \absj{\nuhat_{b_{\max}}^2-\sigma_{\Pp}^2}\le 2e\sqrt{6(\kappa+3)}\sigma_{\Pp}^2\sqrt{\frac{b_{\max}}{n}}~.\]
The second inequality gives
\[\sqrt{1-2e\sqrt{6(\kappa_{\Pp}+3)}\sqrt{\frac{b_{\max}}{n}}}\le \frac{\nuhat_{b_{\max}}}{\sigma_{\Pp}}\le \sqrt{1+2e\sqrt{6(\kappa_{\Pp}+3)}\sqrt{\frac{b_{\max}}{n}}}\]
Since we can assume that 
\[2e\sqrt{6(\kappa_{\Pp}+3)}\sqrt{\frac{b_{\max}}{n}}\le 1~,\]
we deduce that
\[\absj{\nuhat_{b_{\max}}-\sigma_{\Pp}}\le e\sqrt{12(\kappa_{\Pp}+3)}\sigma_{\Pp}\sqrt{\frac{2b_{\max}}{n}}~.\]
This means that, with probability at least $1-2e^{-b_{\max}}$, $(\muhat_n,\Rhat_n)$ belongs to $\sR$ if we define 
\[\epsilon_\mu=2\sqrt{2}\,e\sqrt{\frac{b_{\max}}n}\le \sqrt{\kappa_{\Pp}},\quad \epsilon_R=2e\sqrt{3(\kappa_{\Pp}+3)}\le 19\sqrt{\kappa_{\Pp}}~.\]
By an appropriate choice of the constant $C$, we can always assume that $\sqrt{n/(2b_{\max})\,\kappa_{\Pp}}$ is at least some large constant, to ensure that $2(\epsilon_{\mu}+\epsilon_R)\leq \sqrt{n/(2b_{\max})}$. So \corref{AlmostThere} applies and gives
\[\Pr{\absj{\Pphat_n\Psi_{\muhat_n,\Rhat_n}-\mu_{\Pp}}\le \frac{\sigma_{\Pp}}{\sqrt{n}}\paren{\sqrt{2t}(1+\xi_1)+\xi_2}}\ge 1-2e^{-t}-4e^{-b_{\max}}~,\]
 where
 \[\xi_1=36\sqrt{\frac{\kappa_{\Pp}\,b_{\max}}{n}},\quad \xi_2=2\sqrt2\kappa_{\Pp}\frac{b_{\max}^{3/2}}{n}+ 1120\sqrt{\kappa_{\Pp}}\frac{b_{\max}}{n}~.\]
In particular, if $\delta>\frac{4e}{e-2}e^{-b_{\max}}$, we get
\begin{multline*}
\Pr{\absj{\Pphat_n\Psi_{\muhat_n,\Rhat_n}-\mu_{\Pp}}\le \frac{\sigma_{\Pp}}{\sqrt{n}}\paren{\sqrt{2(1+\ln(1/\delta))}(1+\xi_1)+\xi_2}}\\
\ge 1-\paren{\frac2e+\frac4{4e/(e-2)}}\delta=1-\delta~.
\end{multline*}

\begin{remark}
\label{rem:smallerdeltamin}
Let us quickly sketch how one may get a smaller value of
$\delta_{\min}$ at the expense of a larger constant $L$. The idea is
to redo the proof of part $1$ of \thmref{secondmoment} (cf.\
\secref{variance}). We build \single~estimators for $\mu_{\Pp}$ via
median-of-means, as in \eqnref{almostconfidence}, but then use the
value $2\shat_{b}(X_1^n)$ from \lemref{MoMVariance} instead of the
value $\sigma^2_2$ when building the confidence interval, with a
choice of $b\approx\ln(1/\delta)$. Then one obtains an empirical confidence interval that contains $\mu_{\Pp}$ and has the appropriate length with probability $\geq 1-2\delta$ whenever $\ln(1/\delta)\leq c\,n/\kappa$ for some constant $c>0$. Using \thmref{confidence} as in \secref{variance} then gives a \mult~$L$-sub-Gaussian estimator for $(\sP_{{\rm krt}\leq \kappa},n,e^{1-cn/\kappa})$ for large enough values of $n/\kappa$, where $L$ does not depend on $n$ or $\kappa$. It is an open question whether one can obtain a similar value of $\delta_{\min}$ with $L=\sqrt{2}+\liloh{1}$.\end{remark}

\section{Open problems}\label{sec:open}
  
We conclude the paper by a partial list of problems related to our results that seem especially interesting.

\noindent\par{\bf Sharper constants and truly sub-Gaussian estimators.}  For what families $\sP$
of distributions and what values of $\delta_{\min}$ can one find
\mult~estimators with sharp constant $L=\sqrt{2}+\liloh{1}$? 
One may even sharpen our definition of a sub-Gaussian estimator and
ask for estimators that satisfy
$$\Pr{|\Ehat_n(X_1^n)-\mu_{\Pp}|>\sigma_{\Pp}\,\frac{\Phi^{-1}(1-\delta/2)}{\sqrt{n}}}\leq (1+\liloh{1})\,\delta$$
for all $\Pp\in\sP$ and $\delta\in[\delta_{\min},1)$? \\

\noindent\par{\bf Sub-Gaussian confidence intervals.} 
The notion of sub-Gaussian confidence interval introduced in \secref{confidence} seems interesting on its own right. For which classes of distributions $\sP$ can one find sub-Gaussian confidence intervals? Can one reverse the implication in \thmref{confidence}, and build sub-Gaussian confidence intervals from \mult~estimators?\\

\noindent\par{\bf Empirical risk minimization.} Suppose now that the $X_i$ are i.i.d. random variables that live in an arbitrary measurable space and have common distribution $\Pp$. In a prototypical risk minimization problem, one wishes to find an approximate minimum $\widehat{\theta}_n(X_1^n)$ of a functional $\ell(\theta):=\Pp\,f(\theta,X)$ over choices of $\theta\in\Theta$. The usual way to do this is via empirical risk minimization, which consists of minimizing the empirical risk $\ellhat_n(\theta):=\Pphat \,f(\theta,X)$ instead. Under strong assumptions on the family $F:=\{f(\theta,\cdot)\}_{\theta\in\Theta}$ (such as uniform boundedness), the fluctuations of the empirical process $\{(\Pphat_n-\Pp)\,f(\theta,X)\}_{\theta\in\Theta}$ can be bounded in terms of geometric or combinatorial properties of $F$, and this leads to results on empirical risk minimization. However, the strong sub-Gaussian concentration results one may obtain are only available when $F$ has very light tails. 

A natural way to obtain strong sub-Gaussian concentration for heavier-tailed $F$ would be to replace the usual empirical estimates $\Pphat\,f(\theta,X)$ by one of our \mult~sub-Gaussian estimates. This, however, is not straightforward. The usual chaining technique for controlling empirical processes rely on linearity, and our estimators are nonlinear in the sample. Although there are (artificial) ways around this, we do not know of any {\em efficient method} for doing the analogue of empirical risk minimization with our estimators in any nontrivial setting. These difficulties were overcome by Brownlees et al. \cite{BrownleesEtAlHeavyTailed} via Catoni's \mult~subexponential estimator, at the cost of obtaining weaker concentration. Can one do something similar and achieve truly sub-Gaussian results at low computational cost?

\section{Acknowledgements}

Luc Devroye was supported by the Natural Sciences and Engineering Research Council (NSERC) of Canada.
G\'abor Lugosi and Roberto Imbuzeiro Oliveira gratefully acknowledge support from CNPq, Brazil
via the {\em Ci\^{e}ncia sem Fronteiras} grant \# 401572/2014-5. 
G\'abor Lugosi was supported by the Spanish Ministry of Science and Technology grant MTM2012-37195.
Roberto Imbuzeiro Oliveira's work was supported by a {\em Bolsa de Produtividade em Pesquisa} from CNPq. His work in this article is part of the activities of FAPESP Center for Neuromathematics (grant\# 2013/ 07699-0 , FAPESP - S.Paulo Research Foundation).

\bibliography{rimfo_bibfile}
\bibliographystyle{plain}

\end{document}